\documentclass[reqno,12pt]{amsart} 
\usepackage{vmargin}
\setpapersize{A4}
\usepackage{amsmath} 
\usepackage{amsfonts}   
\usepackage{amssymb}      
\usepackage{eufrak}      
\usepackage{tikz}
\usetikzlibrary{matrix,calc}

\usepackage{amscd}      
\usepackage{amsthm}      
\usepackage{tikz,amsmath}
\usetikzlibrary{arrows}

\usepackage{epsfig}      
\usepackage{amstext}      
\usepackage{graphicx}

\usepackage[all,line,dvips]{xy} 
\CompileMatrices %goer at xy-matricerne koerer hurtigere.

\newcommand{\Aut}{\operatorname{Aut}}

\newcommand{\Span}{\operatorname{Span}}

\newcommand{\Sp}{\operatorname{Sp}}

   \theoremstyle{plain}%default
   \newtheorem{thm}{Theorem}[section]
   \newtheorem{prop}[thm]{Proposition}
   \newtheorem{lemma}[thm]{Lemma}  
   \newtheorem{cor}[thm]{Corollary}
   \theoremstyle{definition}

   \newtheorem{example}[thm]{Example}
   \theoremstyle{remark}
   
   \newtheorem{remark}[thm]{Remark}

\usepackage{mdframed,xcolor}

\definecolor{mybgcolor}{gray}{0.8}
\definecolor{myframecolor}{rgb}{.647,.129,.149}

\mdfdefinestyle{mystyle}{
  usetwoside=false,
  skipabove=0.6em plus 0.8em minus 0.2em,
  skipbelow=0.6em plus 0.8em minus 0.2em,
  innerleftmargin=.25em,
  innerrightmargin=0.25em,
  innertopmargin=0.25em,
  innerbottommargin=0.25em,
  leftmargin=-.75em,
  rightmargin=-0em,
  topline=false,
  rightline=false,
  bottomline=false,
  leftline=false,
  backgroundcolor=mybgcolor,
  splittopskip=0.75em,
  splitbottomskip=0.25em,
  innerleftmargin=0.5em,
  leftline=true,
  linecolor=myframecolor,
  linewidth=0.25em,
}

\newmdenv[style=mystyle]{important}

\usepackage{lipsum}

   \numberwithin{equation}{section}

        \date{\today}
\title[KMS weights]{KMS weights on groupoid and graph $C^*$-algebras}  
\author{Klaus Thomsen}

\date{\today}

\email{matkt@imf.au.dk}
\address{Institut for Matematik, Aarhus University, Ny Munkegade, 8000 Aarhus C, Denmark}

\begin{document}

\maketitle

%\documentclass[a4paper]{memoir}
%\begin{document}

% filen x holder data til foerste tegning
% x= og y= angiver hvilken skala koordinaterne henviser til

% filen y holder data til de to sidste tegninger, eneste forskel
% mellem de to er jo bare placeringen af koordinatlinierne

%\begin{tikzpicture}[x=4cm,y=4cm]
%  \draw[->] (0,0) -- (1.1,0);
%  \draw[->] (0,-0.7) -- (0,0.7);
%  \draw (0.02,0.5) --++ (-0.04,0) node [left] {$\tfrac12$};
%  \draw (0.02,-0.5) --++ (-0.04,0) node [left] {$-\tfrac12$};
%  \draw[thick,line join=round] plot file {y.dat};  
%\end{tikzpicture}

%\end{document}
\begin{abstract} The paper contains a description of the KMS weights for the one-parameter action on the
  reduced $C^*$-algebra of a second countable locally compact
  Hausdorff \'etale groupoid, arising from a continuous
  real valued homomorphism satisfying two conditions. The result is subsequently applied to identify the KMS weights for the gauge action on a simple graph
  algebra. The von Neumann algebra generated by the GNS-representation
  of an extremal $\beta$-KMS weight is a factor, and tools are developed
  to determine its type. The paper concludes with
  three examples to illustrate the results.  
\end{abstract}

\section{Introduction}

The existence and uniqueness of the KMS state for the gauge action on
the $C^*$-algebra of a finite irreducible graph was proved by Enomoto,
Fujii and Watatani, \cite{EFW}, following the result of Olesen and
Pedersen concerning the Cuntz algebra $O_n$, cf. \cite{OP}. More
recently, the KMS states for the gauge action on the graph algebra of
a general finite graph, as well as on its Toeplitz extension, has been
studied in \cite{EL}, \cite{KW} and \cite{aHLRS}. The present work started from the wish to extend these
results to the gauge action on the $C^*$-algebra of an infinite graph, but it soon became
clear that there are typically many more KMS weights than KMS states since the $C^*$-algebra of an
infinite graph is not unital, and often stable. It is therefore 
more natural to look for weights rather than states, and to consider
KMS states, when they exist, as special KMS weights. This
is the point of view taken here, and we obtain a complete description
of the KMS weights for the gauge action on a simple graph $C^*$-algebra of a
row-finite graph without sinks. To obtain this we consider the graph
algebra as a groupoid $C^*$-algebra as in the paper by Kumjian,
Pask, Raeburn and Renault, \cite{KPRR}, where $C^*$-algebras of
infinite graphs were first introduced.

The KMS states for quite
general cocycle actions on the $C^*$-algebra of an \'etale groupoid were
described by Neshveyev in \cite{N}, extending the work by Renault in
\cite{Re1}, and the key result in the present work is a partial
extension to weights of the results of Neshveyev. Specifically, we consider a
second countable locally compact Hausdorff \'etale groupoid $\mathcal
G$ and two continuous homomorphisms (sometimes called cocycles), $c : \mathcal G \to \mathbb R$ and
$c_0 : \mathcal G \to \mathbb R$. Such homomorphisms induce continuous one-parameter
groups, $\sigma^c$ and $\sigma^{c_0}$, of automorphisms on the reduced groupoid $C^*$-algebra $C_r^*(\mathcal G)$ by a
canonical construction, \cite{Re1}. Under the assumption that $\ker c_0$ has
trivial isotropy groups and is open in $\mathcal G$, it is shown in
Theorem \ref{e1} below 
that there is a bijective correspondence between proper
$\sigma^{c_0}$-invariant KMS
weights for $\sigma^c$ and regular Borel measures on the unit space $\mathcal G^{(0)}$ of $\mathcal G$ that satisfy a
certain conformality condition introduced by Renault in \cite{Re1}. The KMS
weight $\varphi_m$ corresponding to such a measure $m$ is defined by
the expression
$$
\varphi_m(a) = \int_{\mathcal G^{(0)}} P(a) \ dm,
$$
where $P : C^*_r(\mathcal G) \to C_0\left(\mathcal G^{(0)}\right)$ is
the canonical conditional expectation. The proof of this, as well as
the entire approach to KMS weights, builds
on the theory developed with locally compact quantum
groups in mind, by Kustermans and Vaes in \cite{Ku}, \cite{KV2} and \cite{KV2}.

The result is then applied in a relatively straightforward way to the Renault, Deaconu, Anantharaman-Delaroche groupoid,
\cite{Re1},\cite{De}, \cite{An}, or RDA-groupoid for short, arising
from a local homeomorphism on a locally compact second countable
Hausdorff space, giving a bijective correspondence between measures and
gauge-invariant KMS weights for a general cocycle action. Since the
graph algebra of a countable row-finite graph without sinks is the
RDA-groupoid of the shift on the locally compact Hausdorff space of
infinite paths in the graph, we can then subsequently specialize to
graph algebras. 

Let $G$ be a countable row-finite graph without sinks. 
%Let $V_G$ be the
%set of
%vertexes in $G$ and $A = (A_{vw}}), v,w \in V_G$, the adjacency
%matrix. 
We show in
Theorem \ref{THM2} below that for any $\beta  \in \mathbb R$ there is
a bijective correspondence between the $\beta$-KMS weights on $C^*(G)$
and positive $e^{\beta}$-eigenvectors for the adjacency matrix of
$G$. The $\beta$-KMS states correspond to such eigenvectors whose
coordinates sum to 1; a conclusion which has also recently been
obtained by Toke M. Carlsen and Nadia Larsen, \cite{Ca}, by a different method. Thus the search for the KMS weights boils
down to an interesting eigenvalue problem for possibly infinite
non-negative matrices, the solution of which is described for cofinal
matrices in the accompanying paper \cite{Th3}. For irreducible
matrices much of the story on the solutions was already known or could
be derived from known results, some
parts from the results of Pruitt, \cite{P}, and Vere-Jones, \cite{V},
and other parts from the theory of countable state Markov chains,
cf. e.g. \cite{Wo}. It turns out that the set of $\beta$-values for
which there is a $\beta$-KMS weight for the gauge action first of all
depends on what we here call the non-wandering part of $G$, by which
we mean the set $NW_G$ of vertexes that are contained in a loop in the
graph. When the graph is cofinal these vertexes and the edges they
emit constitute an irreducible sub-graph of $G$, and we let $\beta_0$
be the exponential growth rate of the number of loops based at a vertex
in $NW_G$. This does not depend on the vertex and it can be any element
of $[0,\infty]$. Our main result concerning KMS weights on graph
algebras is the following.     

\begin{thm}\label{betavalues(intro)}  Assume that
  $G$ is cofinal and let $\beta \in \mathbb R$.
\begin{enumerate}
\item[1)] Assume that the non-wandering part $NW_G$ is empty. There is
  a $\beta$-KMS weight for the gauge action on $C^*(G)$ for all $\beta \in \mathbb R$. 
\item[2)]  Assume that the non-wandering part $NW_G$ is non-empty and
  finite. There is a $\beta$-KMS weight for the gauge action on
  $C^*(G)$ if and only if $\beta = \beta_0$.
\item[3)]  Assume that the non-wandering part $NW_G$ is non-empty and
  infinite. There is a $\beta$-KMS weight for the gauge action on
  $C^*(G)$ if and only if $\beta \geq \beta_0$.
\end{enumerate}
\end{thm} 

By combining the results from \cite{Th3} with those of this paper, it is
also possible to obtain a description of the
corresponding KMS weights, at least in principle.

In Section \ref{k40} we study the $\Gamma$-invariant of Connes,
\cite{C1}, for the factor
$\pi_{\psi}\left(C^*(G)\right)''$ generated by the
GNS-representation $\pi_{\psi}$ of an extremal $\beta$-KMS weight $\psi$. When
$G$ is cofinal with uniformly bounded out-degree we show that there
are natural numbers, $d'_G$ and $d_G$, both generalizing the period of a
finite irreducible graph, such that
$$
\mathbb Zd'_G \beta \subseteq \Gamma\left(\pi_{\psi}(C^*(G))''\right)
\subseteq \mathbb Zd_G\beta .
$$ 
In some cases the numbers $d'_G$ and $d_G$ are the same, but not in
general. We leave it as an open problem to find a method
to determine the $\Gamma$-invariant when they differ.

We conclude the paper with three examples to illustrate the results,
and in particular demonstrate the significance of considering weights,
and not only states, in relation to the KMS condition.

\section{Measures and KMS weights on groupoid $C^*$-algebras}

Let $A$ be a $C^*$-algebra and $A_{+}$ the convex cone of positive
elements in $A$. A \emph{weight} on $A$ is map $\psi : A_{+} \to
[0,\infty]$ with the properties that $\psi(a+b) = \psi(a) + \psi(b)$
and $\psi(\lambda a) = \lambda \psi(a)$ for all $a, b \in A_{+}$ and all
$\lambda \in \mathbb R, \ \lambda > 0$. By definition $\psi$ is
\emph{densely defined} when $\left\{ a\in A_{+} : \ \psi(a) <
  \infty\right\}$ is dense in $A_{+}$ and \emph{lower semi-continuous}
when $\left\{ a \in A_{+} : \ \psi(a) \leq \alpha \right\}$ is closed
for all $\alpha \geq 0$. We will use \cite{Ku}, \cite{KV1} and \cite{KV2} as our source for
information on weights, and as in \cite{KV2} we say that a weight is \emph{proper}
when it is non-zero, densely defined and lower semi-continuous. 

Let $\psi$ be a proper weight on $A$. Set $\mathcal N_{\psi} = \left\{ a \in A: \ \psi(a^*a) < \infty
\right\}$ and note that 
\begin{equation*}\label{f3}
\mathcal N_{\psi}^*\mathcal N_{\psi} = \Span \left\{ a^*b : \ a,b \in
  \mathcal N_{\psi} \right\}
\end{equation*} 
is a dense
$*$-subalgebra of $A$, and that there is a unique well-defined linear
map $\mathcal N_{\psi}^*\mathcal N_{\psi} \to \mathbb C$ which
extends $\psi : \mathcal N_{\psi}^*\mathcal N_{\psi} \cap A_+ \to
[0,\infty)$. We denote also this densely defined linear map by $\psi$.

Let $\alpha : \mathbb R \to \Aut A$ be a point-wise
norm-continuous one-parameter group of automorphisms on
$A$. Let $\beta \in \mathbb R$. Following \cite{C} we say that a proper weight
$\psi$ on $A$ is a \emph{$\beta$-KMS
  weight} for $\alpha$ when
\begin{enumerate}
\item[i)] $\psi \circ \alpha_t = \psi$ for all $t \in \mathbb R$, and
\item[ii)] for every pair $a,b \in \mathcal N_{\psi} \cap \mathcal
  N_{\psi}^*$ there is a continuous and bounded function $F$ defined on
  the closed strip $D_{\beta}$ in $\mathbb C$ consisting of the numbers $z \in \mathbb C$
  whose imaginary part lies between $0$ and $\beta$, and is
  holomorphic in the interior of the strip and satisfies that
$$
F(t) = \psi(a\alpha_t(b)), \ F(t+i\beta) = \psi(\alpha_t(b)a)
$$
for all $t \in \mathbb R$. \footnote{Note that we apply the
  definition from \cite{C} for the action $\alpha_{-t}$
  in order to use the same sign convention as in \cite{BR}, for example.}
\end{enumerate}   
A $\beta$-KMS weight $\psi$ with the property that 
$$
\sup \left\{ \psi(a) : \ 0 \leq a \leq 1 \right\} = 1
$$
will be called a \emph{$\beta$-KMS state}. This is consistent with the
standard
definition of KMS states, \cite{BR}, except when $\beta = 0$ in which
case our definition requires also that a $0$-KMS state, which is also a trace state, is $\alpha$-invariant.

In this section we investigate KMS weights for a specific class of
one-parameter groups of automorphisms on the $C^*$-algebra of an
\'etale groupoid. To introduce these algebras, let $\mathcal G$ be an \'etale second countable locally compact Hausdorff groupoid with unit
space $\mathcal G^{(0)}$. Let $r : \mathcal G \to \mathcal G^{(0)}$ and $s : \mathcal G \to \mathcal G^{(0)}$ be
the range and source maps, respectively. For $x \in \mathcal G^{(0)}$ put $\mathcal G^x = r^{-1}(x), \ \mathcal G_x = s^{-1}(x) \ \text{and} \ \mathcal G^x_x =
s^{-1}(x)\cap r^{-1}(x)$. Note that $\mathcal G^x_x$ is a group, the \emph{isotropy group} at $x$. The space $C_c(\mathcal G)$ of continuous compactly supported
functions is a $*$-algebra when the product is defined by
\begin{equation*}\label{k15}
(f_1 * f_2)(g) = \sum_{h \in \mathcal G^{r(g)}} f_1(h)f_2(h^{-1}g)
\end{equation*}
and the involution by $f^*(g) = \overline{f\left(g^{-1}\right)}$.
To define the \emph{reduced groupoid
  $C^*$-algebra} $C^*_r(\mathcal G)$, let $x\in \mathcal G^{(0)}$. There is a
representation $\pi_x$ of $C_c(\mathcal G)$ on the Hilbert space $l^2(\mathcal G_x)$ of
square-summable functions on $\mathcal G_x$ given by 
\begin{equation*}\label{pix}
\pi_x(f)\psi(g) = \sum_{h \in \mathcal G^{r(g)}} f(h)\psi(h^{-1}g) .
\end{equation*}
$C^*_r(\mathcal G)$ is the completion of $C_c(\mathcal G)$ with respect to the norm
$$
\left\|f\right\|_r = \sup_{x \in \mathcal G^{(0)}} \left\|\pi_x(f)\right\| .
$$  
Note that $C^*_r(\mathcal G)$ is separable since we assume that the
topology of $\mathcal G$ is second countable.

The map $C_c(\mathcal G) \to C_c\left(\mathcal G^{(0)}\right)$ which
restricts functions to $\mathcal G^{(0)}$ extends to a conditional expectation $P
: C^*_r(\mathcal G) \to C_0\left(\mathcal G^{(0)}\right)$. Via $P$ a
regular Borel measure $m$ on $\mathcal G^{(0)}$ gives rise to a weight
$\varphi_m : C^*_r(\mathcal G)_+ \to [0,\infty]$ defined by the
formula
\begin{equation*}\label{f1}
\varphi_m(a) = \int_{\mathcal G^{(0)}} P(a) \ dm .
\end{equation*}
It follows from Fatou's lemma that $\varphi_{m}$ is lower
semi-continuous. Since $\varphi_{m}(f a f) < \infty$ for every
non-negative function $f$ in $C_c(\mathcal G^{(0)})$, it follows that $\varphi_m$
is also densely defined, i.e. $\varphi_m$ is a proper weight on
$C^*_r(\mathcal G)$.\footnote{We consider only non-zero measures.}

Let $c : \mathcal G \to \mathbb R$ be a continuous homomorphism, i.e. $c$ is
continuous and $c(gh) = c(g) + c(h)$ when $s(g) = r(h)$. For each $t
\in \mathbb R$ we can then define an automorphism $\sigma_t^c$ of
$C_c(\mathcal G)$ such that
\begin{equation}\label{b312}
\sigma^c_t(f)(g) = e^{itc(g)}f(g) .
\end{equation} 
For each $x \in \mathcal G^{(0)}$ the same expression defines a unitary $u_t$ on
$l^2(\mathcal G_x)$ such that $u_t\pi_x(f)u_t^* =
\pi_x\left(\sigma^c_t(f)\right)$ and it follows therefore that
$\sigma^c_t$ extends by continuity to an automorphism $\sigma^c_t$ of
$C^*_r\left(\mathcal G\right)$. It is easy to see that $\sigma^c =
\left(\sigma^c_t\right)_{t \in \mathbb R}$ is a continuous
one-parameter group of automorphisms on
$C^*_r\left(\mathcal G\right)$. The $*$-subalgebra $C_c(\mathcal G)$ of $C^*_r(\mathcal G)$ consists of
elements that are analytic for $\sigma^c$, cf. \cite{BR}. Since $P(C_c(\mathcal G))
\subseteq C_c\left(\mathcal G^{(0)}\right)$ we see that $C_c(\mathcal G) \subseteq \mathcal N_{\varphi_m}^*\mathcal
N_{\varphi_m}$ for every regular Borel measure $m$ on $\mathcal G^{(0)}$. Let $\beta \in \mathbb R$. As in \cite{Th3} we say that $m$ is   
 \emph{$(\mathcal G,c)$-conformal with exponent $\beta$}
when
\begin{equation}\label{c31} 
m(s(W)) = \int_{r(W)} e^{\beta c(r_W^{-1}(x))} \ dm(x)
\end{equation}
for every open bi-section $W \subseteq \mathcal G$, where $r_W^{-1}$
denotes the inverse of $r: W \to r(W)$. This terminology is motivated
by the resemblance with the notion of conformality for measures used
for dynamical systems, cf. \cite{DU}. In certain cases the notions
actually coincide, cf. Lemma \ref{d80} below.

The relation between $(\mathcal G,c)$-conformality and KMS weights is given by the following

\begin{prop}\label{c11}% Consider the one-parameter group $\sigma^c$ on
%  $C^*_r(\mathcal G)$ arising from a continuous homomorphism $c :\mathcal G \to \mathbb R$. Let $\beta \in \mathbb R \backslash
%  \{0\}$ and 
Let $m$ be a regular Borel measure on $\mathcal
  G^{(0)}$. The following are equivalent: 
\begin{enumerate}
\item[1)] $m$ is
  $(\mathcal G,c)$-conformal with exponent $\beta$.
\item[2)] \begin{equation*}\label{c9}
\varphi_{m}(fg) = \varphi_{m}\left(g\sigma^c_{i\beta}(f)\right)
\end{equation*} 
for all $f,g \in C_c(\mathcal G)$.
\item[3)] $\varphi_{m}$ is a $\beta$-KMS weight for $\sigma^c$.
\end{enumerate}
\end{prop} 
\begin{proof} The equivalence of 1) and 2) follows from a calculation first performed by
  Renault on page 114 in \cite{Re1} and later extended by Neshveyev in
  the proof of Theorem 1.3 in \cite{N}. We will not repeat it here.

2) $\Rightarrow$ 3): Consider elements $a,b \in \mathcal
N_{\varphi_{m}} \cap \mathcal N_{\varphi_{m}}^*$. Let $\{g_k\}$ be
an approximate unit for $C^*_r(\mathcal G)$ consisting of elements
from $C_c(\mathcal G^{(0)})$ and choose
sequences $\{a_n\},\{b_n\}$ in $C_c(\mathcal G)$ converging to $a$ and $b$,
respectively. For $k,n \in \mathbb N$ consider the entire function 
$$
F_{k,n}(z) = \varphi_{m}\left(g_ka_n g_k\sigma_z^c(b_n)\right) .
$$ 
Note that by assumption $F_{k,n}(t + i \beta) = \varphi_{m}\left(g_ka_n
  \sigma_{i\beta}^c\left(\sigma^c_t(g_kb_n)\right)\right) =
\varphi_{m}\left(g_k \sigma^c_t(b_n)g_ka_n\right)$. Since $P(cd) \leq P(cc^*)^{\frac{1}{2}}P(d^*d)^{\frac{1}{2}}$ for all $c,d
\in C^*_r(\mathcal G)$ and $P \circ \sigma^c_t = P$, it follows that
\begin{equation*}\label{c12}
\begin{split}
& \sup_{t \in \mathbb R} \left|F_{k,n}(t) -
  \varphi_{m}(g_kag_k\sigma^c_t(b))\right| \\
&\leq \sup_{t \in \mathbb
  R} \left(\left|\varphi_{m}\left( g_k(a_n-a)g_k
      \sigma^c_t(b_n)\right)\right| + \left|\varphi_{m}(g_kag_k
    \sigma^c_t(b_n-b))\right|\right)\\
& \leq \int_{\mathcal G^{(0)}} g_k \left(P((a_n-a)(a_n-a)^*)^{\frac{1}{2}} P(b_n^*g_k^2
b_n)^{\frac{1}{2}} + P(aa^*)^{\frac{1}{2}}P\left((b_n-b)^*g_k^2
  (b_n-b)\right)^{\frac{1}{2}}\right) \ dm .
\end{split} 
\end{equation*}
It follows that 
\begin{equation}\label{f7}
\lim_{n \to \infty} \sup_{t \in \mathbb R} \left|F_{k,n}(t) -
  \varphi_{m}(g_kag_k\sigma^c_t(b))\right| = 0. 
\end{equation}
Similar estimates
show that 
\begin{equation}\label{f8}
\lim_{n \to \infty} \sup_{t \in \mathbb R} \left|F_{k,n}(t+i\beta ) -
  \varphi_{m}(g_k\sigma^c_t(b)g_ka)\right| = 0.
\end{equation}
For each $k,n$ the function $F_{k,n}$ is bounded by
$$
\sup_{ |s| \leq |\beta|} \sqrt{\varphi_m\left(g_ka_na_n^*g_k\right)
  \varphi_m\left(\sigma^c_{is}(b_n)^*\sigma^c_{is}(b_n)\right)} 
$$
on the closed strip $D_{\beta}$ in
$\mathbb C$ where the imaginary part is between $0$ and $\beta$. We
can therefore use Hadamard's three-lines theorem, Proposition 5.3.5 in
\cite{BR}, to conclude from
(\ref{f7}) and (\ref{f8}) that the sequence
$\left\{F_{k,n}\right\}_{n=1}^{\infty}$ is Cauchy in the supremum norm
on $D_{\beta}$ and therefore
converges uniformly, as $n$ tends to $\infty$, to a continuous and bounded
function $F_k : D_{\beta} \to \mathbb C$ which is holomorphic in the
interior of $D_{\beta}$ and satisfies that
$$
F_k(t) = \varphi_{m}(g_kag_k\sigma^c_t(b)) \ \text{and} \
F_k(t+i\beta) = \varphi_{m}(g_k\sigma^c_t(b)g_ka)
$$
for all $t \in \mathbb R$. Now note that
\begin{equation*}\label{c16}
\begin{split}
& \left|F_k(t) - \varphi_{m} (a\sigma^c_t(b))\right| \leq
\left|\varphi_{m}((g_ka-a)g_k\sigma^c_t(b))\right| +
\left|\varphi_{m}(a(g_k\sigma^c_t(b) -\sigma^c_t(b))\right|\\
& \leq \int_{\mathcal G^{(0)}} \left|g_k-1\right| P(aa^*)^{\frac{1}{2}}
P(b^*b)^{\frac{1}{2}} + P(aa^*)^{\frac{1}{2}}
P(b^*(g_k-1)^2b)^{\frac{1}{2}} \ dm .
\end{split}
\end{equation*}
The integrand is dominated by
$2P(aa^*)^{\frac{1}{2}}P(b^*b)^{\frac{1}{2}}$ which is in
$L^1(m)$ because $a$ and $b$ are both in $\mathcal N_{\varphi_{m}}
\cap \mathcal N_{\varphi_{m}}^*$. We can therefore apply Lebesgue's
dominated convergence theorem to conclude that 
$$
\lim_{k \to \infty} \sup_{t \in \mathbb R} \left|F_k(t) -
  \varphi_{m} (a\sigma^c_t(b))\right| = 0.
$$
Similar arguments show that
$$
\lim_{k \to \infty} \sup_{t \in \mathbb R} \left|F_k(t+i\beta) -
  \varphi_{m} (\sigma^c_t(b)a)\right| = 0,
$$
and 3) follows then from Hadamard's three-lines theorem as above.

3) $\Rightarrow$ 2): Since $C_c(\mathcal G) \subseteq \mathcal
N_{\varphi_{m}} \cap \mathcal N_{\varphi_{m}}^*$ this
implication is obtained exactly as for states. See the proof of (3)
$\Rightarrow$ (1) in Proposition 5.3.7 of \cite{BR}.

\end{proof}

The key result of the paper is the following theorem. It has a
predecessor in Proposition 5.4 of \cite{Re1}, but the two results are
not directly comparable because the definitions of KMS weights are not
the same. In particular, a major problem overcome in the following
proof is that with the present definition, KMS weights are not a priori finite on $C_c(\mathcal G)$.

\begin{thm}\label{e1} Let $c_0 : \mathcal G \to \mathbb
R$ and $c : \mathcal G \to \mathbb R$ be continuous homomorphisms. Assume that 
\begin{enumerate}
 \item[a)] $\ker c_0 = \left\{ g \in \mathcal G : \ c_0(g) = 0\right\}$ is an
  equivalence relation, i.e. $\mathcal G^x_x \cap \ker c_0 = \{x\}$ for all $x
  \in \mathcal G^{(0)}$, and 
\item[b)] $\ker c_0$ is open in $\mathcal G$.
 \end{enumerate}
Let $\psi$ be a $\beta$-KMS weight for
  the action $\sigma^{c}$ on $C^*_r(\mathcal G)$. Assume that $\psi$ is invariant
  under $\sigma^{c_0}$ in the sense that $\psi \circ \sigma^{c_0}_t  =
  \psi$ for all $t \in \mathbb R$. 

There is a regular Borel measure $m$ on $\mathcal G^{(0)}$ such that $\psi =
  \varphi_{m}$.
\end{thm}
\begin{proof} By the Riesz representation theorem it suffices to show
  that $\psi$ takes finite values on non-negative elements from
  $C_c(\mathcal G^{(0)})$ and that $\psi \circ P = \psi$.

By assumption $\ker c_0$ is an open
(and closed) 
sub-groupoid of $\mathcal G$ and hence $C^*_r(\ker c_0) \subseteq C^*_r(\mathcal G)$. For
$R > 0$, set
\begin{equation}\label{e2}
Q_R(a) = \frac{1}{R}\int_0^{R} \sigma^{c_0}_t(a) \ d t.
\end{equation}
It follows easily from the formula defining the action $\sigma^c$, cf.
(\ref{b312}), that
$$
\lim_{R \to \infty} Q_R(f)(g) = \begin{cases} f(g) & \ \text{when} \ g
  \in \ker c_0 \\ 0 & \ \text{otherwise} \end{cases} 
$$
for every $f \in C_c(\mathcal G)$.
Since $\left\|Q_R\right\| \leq 1$ it follows therefore that the limit
$$
Q(a) = \lim_{R \to \infty} Q_R(a)
$$
exists for all $a \in C^*_r(\mathcal G)$ and that $Q(a) \in
C^*_r\left(\ker c_0\right)$. We aim to prove that $\psi \circ Q_R =
\psi$ and $\psi \circ Q \leq \psi$. To this end observe that there is a directed subset $\Lambda$ of 
$$
\left\{ \omega \in C^*_r\left(\mathcal G\right)^*
  : \ 0 \leq \omega \leq \psi \right\}
$$ 
such that $\psi(a) = \lim_{\omega \in
  \Lambda} \omega(a)$
for all $a \geq 0$, cf. \cite{Ku}. Fix $R > 0$ and let $a\geq 0$. Assume first that
$\psi(a) = \infty$. For any $r > 0$ and any
$t \in [0,R]$ there is an $\omega_{t} \in \Lambda$ such
that $\omega_{t}(\sigma^{c_0}_{s}(a)) > r$ for all $s$ in an open neighborhood of $t$. Thanks to the directedness of
$\Lambda$ there is therefore an $\omega \in \Lambda$ such that
$\omega\left(\sigma^{c_0}_{t}(a)\right) \geq r$ for all $t \in [0,R]$. Then  
$$
\psi(Q_R(a))\geq \omega(Q_R(a)) =  \frac{1}{R}\int_0^R
\omega(\sigma^{c_0}_{t}(a)) \ dt \ 
\geq \ r,
$$
proving that $\psi(Q_R(a)) = \infty$. If instead $\psi(a) < \infty$ we
let $\epsilon > 0$ and choose in the same way an $\omega \in \Lambda$ such
that $\omega(\sigma^{c_0}_{t}(a)) \geq \psi(a) - \epsilon$ for all $t \in [0,R]$. Then
\begin{equation}\label{e3}
\psi(Q_R(a)) \geq \omega(Q_R(a)) =  \frac{1}{R} \int_0^R
\omega(\sigma^{c_0}_{t}(a))  \ d t \geq \psi(a) - \epsilon .
\end{equation}
On the other hand the set $\left\{ b \in C^*_r\left(\mathcal G\right): \ b \geq 0, \ \psi(b) \leq
  \psi(a) \right\}$
is closed, convex and contains all Riemann sums for the
integral in (\ref{e2}), and it follows
therefore that $\psi(Q_R(a)) \leq \psi(a)$. Combined with (\ref{e3})
this shows that $\psi(Q_R(a)) = \psi(a)$ also when $\psi(a) < \infty
$. Thus
\begin{equation}\label{e4}
\psi \circ Q_R = \psi.
\end{equation}
The lower semi-continuity of $\psi$ now implies that
\begin{equation}\label{e21}
\psi \circ Q \leq \psi.
\end{equation}

We aim next to show that $\psi(g) < \infty$ for all non-negative $g
\in C_c(\mathcal G^{(0)})$. Let therefore $f$ be such a function. Since $\psi$ is densely
defined there is a sequence $\{a_n\}$ of non-negative elements in
$C^*_r\left(\mathcal G\right)$ such that $\lim_n a_n = \sqrt{f}$ and
$\psi(a_n) < \infty$ for all $n$.  
It
follows that $\lim_n Q(a_n) = \sqrt{f}$. Set $b_n =Q(a_n)$ and note
that $\psi(b_n) \leq  \psi(a_n) < \infty$ by (\ref{e21}). Let $k \in \mathbb N$ and consider
\begin{equation}\label{e5}
c_n =  \sqrt{\frac{k}{\pi}} \int_{\mathbb R}
\sigma^c_t\left(b_n\right) e^{-kt^2} \ dt.
\end{equation}
Then $c_n \in C^*_r\left(\ker c_0\right)$ is analytic for $\sigma^c$ and if $k$ is large enough
$\left\|c_n - b_n\right\| \leq \frac{1}{n}$,
cf. \cite{BR}. By approximating the integral in (\ref{e5}) by convex
combinations of elements of the form $\sigma^c_t(b_n)$ and by using that $\psi \circ \sigma^c_t = \psi$, it
follows from the lower semi-continuity of $\psi$ that $\psi(c_n) \leq
\psi\left(b_n\right) <
\infty$. Since $\ker c_0$ has trivial isotropy by the first condition a) on
$c_0$, it follows from (the
proof of) Lemma 2.24 in \cite{Th1} that there is a sequence 
$$
\left\{ d^n_j
  : \ j = 1,2, \dots, N_n \right\}, \ n = 1,2,3, \dots,
$$ 
of
non-negative elements in $C_0(\mathcal G^{(0)})$ such that 
\begin{equation}\label{e6}
\sum_{j=1}^{N_n} {d^n_j}^2 \leq 1
\end{equation}
for all $n$, and
\begin{equation}\label{e7}
P(b) = \lim_{n \to \infty}  \sum_{j=1}^{N_n} d^n_jb d^n_j  
\end{equation}
for all $b \in C^*_r(\ker c_0)$. By using that $\psi$ is densely defined in combination with
(\ref{e21}), we can choose, for each $d^n_j$ and each $k\in \mathbb N$, a positive
element $c(n,j,k) \in C^*_r(\ker c_0)$ such that $\left\|c(n,j,k)d^n_j
  -d^n_j\right\| \leq \frac{1}{k}$ and $\psi\left(c(n,j,k)\right) <
\infty$. Set 
$$
c'(n,j,k) = \frac{1}{\sqrt{\pi}} \int_{\mathbb R}
e^{-t^2}\sigma^c_t(c(n,j,k)) \ dt .
$$
Then $c'(n,j,k)$ is analytic for $\sigma^c$, $\psi(c'(n,j,k)) \leq
\psi\left(c(n,j,k)\right) < \infty$, and
$$
\left\|c'(n,j,k)d^n_j - d^n_j \right\| = \left\|\frac{1}{\sqrt{\pi}} \int_{\mathbb R}
  \sigma^c_t(c(n,j,k) d^n_j - d^n_j) e^{-t^2} \ dt\right\| \ \leq \
  \frac{1}{k}.
$$
Note that it follows
from the definition of a $\beta$-KMS weight that when $a \in \mathcal N_{\psi} \cap \mathcal
N_{\psi}^*$ is analytic for $\sigma^c$ the following identity holds:
\begin{equation}\label{e8}
\psi\left(a^*a\right) = \psi\left( \sigma^c_{\frac{-i\beta}{2}}(a)\sigma^c_{\frac{-i\beta}{2}}(a)^*\right) .
\end{equation}
(See also Proposition 1.11 in \cite{KV2}.) Since $c'(n,j,k)d_j^nc_m \in \mathcal N_{\psi} \cap \mathcal
N_{\psi}^*$ is analytic for $\sigma^c$ it follows from (\ref{e8}) that
\begin{equation}\label{e9}
\begin{split}
&\psi\left( \sum_{j=1}^{N_n} c'(n,j,k)d^n_j c_m^2 d^n_jc'(n,j,k) \right) = \sum_{j=1}^{N_n}
\psi\left( c'(n,j,k)d^n_j c_m^2 d^n_jc'(n,j,k)\right) \\
%& = \sum_{i=1}^{N_n} \psi\left(
% \sigma^c_{\frac{-i\beta}{2}}\left(c_m
%   {d^n_i}c'(k,i,n)\right) \sigma^c_{\frac{-i\beta}{2}}\left(c_md^n_ic'(k,i,n) %\right)^*\right) \\
& =  \sum_{j=1}^{N_n} \psi\left(
 \sigma^c_{\frac{-i\beta}{2}} (c_m)
   {d^n_j} \sigma^c_{\frac{-i\beta}{2}}(c'(n,j,k)) \sigma^c_{\frac{-i\beta}{2}}(c'(n,j,k))^* d^n_j\sigma_{\frac{-i\beta}{2}}(c_m)^*\right) .
\end{split}
\end{equation}
Because
$$
\lim_{k \to \infty} d^n_j\sigma^c_{\frac{-i\beta}{2}}(c'(n,j,k))
-d^n_j = \lim_{k \to \infty} \frac{1}{\sqrt{\pi}}
\int_{\mathbb R} \sigma^c_t(d^n_jc(n,j,k)  -d^n_j) e^{-
  (t+\frac{i\beta}{2})^2} \ dt = 0
$$  
and $\sigma^c_{\frac{-i\beta}{2}} (c_m) \in
\mathcal N_{\psi}^*$, it follows that
\begin{equation*}
\begin{split}
&\lim_{k \to \infty}\sum_{j=1}^{N_n} \psi\left(
 \sigma^c_{\frac{-i\beta}{2}} (c_m)
   {d^n_j} \sigma^c_{\frac{-i\beta}{2}}(c'(n,j,k))
   \sigma^c_{\frac{-i\beta}{2}}(c'(n,j,k))^*
   d^n_j\sigma_{\frac{-i\beta}{2}}(c_m)^*\right)  \\
& \ \ \ \ \ \ \ = \sum_{j=1}^{N_n} \psi\left(
 \sigma^c_{\frac{-i\beta}{2}} (c_m)
   {d^n_j}^2\sigma_{\frac{-i\beta}{2}}(c_m)^*\right). 
\end{split}
\end{equation*}
We can therefore let $k$ tend to infinity in (\ref{e9}) and use the
lower semi-continuity of $\psi$ in combination with (\ref{e6}) and
(\ref{e8}) to conclude that
\begin{equation*}
\begin{split}
&
\psi\left( \sum_{j=1}^{N_n} d^n_j c_m^2 d^n_j \right) \leq  \sum_{j=1}^{N_n} \psi\left(
 \sigma^c_{\frac{-i\beta}{2}} (c_m)
   {d^n_j}^2\sigma_{\frac{-i\beta}{2}}(c_m)^*\right) \\
&\leq   \psi\left(
 \sigma^c_{\frac{-i\beta}{2}}
 (c_m)\sigma_{\frac{-i\beta}{2}}(c_m)^*\right) = \psi(c_m^2) \leq
\left\|c_m\right\| \psi(c_m).
\end{split}
\end{equation*}
Let then $n$ tend to infinity to conclude that
$\psi\left(P(c_m^2)\right) \leq \left\|c_m\right\|\psi(c_m)$. Since
$\psi(c_m) < \infty$ and $\lim_{m \to
  \infty} P(c_m^2) = f$, we conclude that
$f$ is the limit under uniform convergence of the sequence $g_m =
P(c_m^2)$ of non-negative functions from $C_0(\mathcal G^{(0)})$ such that $\psi(g_m)
< \infty$ for all $m$. 

Consider then an arbitrary non-negative $g \in C_c(\mathcal G^{(0)})$
and let $f$ be a
non-negative function $f \in C_c(\mathcal G^{(0)})$ such that $f(x) =2$ for all $x$ in
the support of $g$. As we have just shown there is a non-negative
function $h \in C_0(\mathcal G^{(0)})$ such that $\left|h(y) - f(y)\right| \leq 1$
for all $y\in \mathcal G^{(0)}$ and $\psi(h) < \infty$. Then $0 \leq g \leq gh$ and
$\psi(g) \leq \psi(gh) < \infty$.

Let then $a \geq 0$ in $C^*_r(\mathcal G)$ and $0\leq g \leq 1$ in
$C_c(\mathcal G^{(0)})$. By the same methods as above we construct a
sequence $\{c_k\}$ of positive $\sigma^c$-analytic elements in $C^*_r\left( \ker c_0\right)$ converging to $\sqrt{Q(a)}$ with $\psi(c_k)
< \infty$ for all $k$. Using that $\psi(g^2) < \infty$ we find that 
\begin{equation*}
\begin{split}
& \psi(gP(c_k^2)g) = \lim_{n\to \infty} \sum_{j=1}^{N_n}
\psi(d^n_j gc_k^2g d^n_j) \\
& = \lim_{n \to \infty} \sum_{j=1}^{N_n} \psi\left(\sigma^c_{\frac{-i\beta}{2}}(c_k) g{d^n_j}^2g \sigma^c_{\frac{-i\beta}{2}}(c_k)^* \right) 
\leq \psi\left(\sigma^c_{\frac{-i\beta}{2}}(c_k)g^2
  \sigma^c_{\frac{-i\beta}{2}}(c_k)^* \right) \\
& =
\psi\left(\sigma^c_{\frac{-i\beta}{2}}(c_kg)\sigma^c_{\frac{-i\beta}{2}}(c_kg)^*
\right) = \psi(gc_k^2g) 
\end{split}
\end{equation*}
for all $k$, and when we let $k$ tend to infinity we deduce that
$\psi(gP(a)g) = \psi(gQ(a)g)$ since $P\circ Q = P$. But $\psi(gQ(a)g) = \lim_{R \to \infty}
\psi(gQ_R(a)g) = \lim_{R \to \infty} \psi\left(Q_R(gag)\right) =
\psi(gag)$, thanks to (\ref{e4}). Thus 
\begin{equation}\label{e19}
\psi(gP(a)g) = \psi(gag).
\end{equation} 
Since
$gP(a)g \leq P(a)$ it follows that 
\begin{equation}\label{e10}
\psi(gag) \leq \psi(P(a)).
\end{equation} 
Let $\left\{g_k\right\}$ be a sequence from $C_c(\mathcal G^{(0)})$ which constitutes an
  approximate unit in $C^*_r\left(\mathcal G\right)$. Inserting
$g_k$ for $g$ in (\ref{e10}), the lower
semi-continuity of $\psi$ yields that
\begin{equation}\label{e11}
\psi(a) \leq \psi(P(a)) .
\end{equation}

Let now $a \geq 0$ be a positive element in
$C^*_r\left(\mathcal G\right)$ such that $\psi(P(a)) <
\infty$. Then $\psi(a) < \infty$ and we can combine the Cauchy-Schwarz
inequality with (\ref{e11}) to get
\begin{equation*}\label{e12}
\begin{split}
&\left|\psi(g_kag_k) - \psi(a)\right| \leq \left|
  \psi\left(g_ka(g_k-1)\right)\right| + \left|\psi((g_k-1)a)\right| \\
& \leq \psi\left(g_kag_k\right)^{\frac{1}{2}}\psi\left((g_k-1)a(g_k-1)\right)^{\frac{1}{2}} +
  \psi\left((g_k-1)a(g_k-1)\right)^{\frac{1}{2}}\psi(a)^{\frac{1}{2}}\\
& \leq \psi\left(g_kP(a)g_k\right)^{\frac{1}{2}}\psi\left((g_k
    -1)P(a)(g_k-1)\right)^{\frac{1}{2}} +
  \psi\left((g_k-1)P(a)(g_k-1)\right)^{\frac{1}{2}}\psi(a)^{\frac{1}{2}}.
\end{split}
\end{equation*}
Note that $\lim_{k \to \infty} \psi\left(g_kP(a)\right) = \lim_{k \to
  \infty} \psi\left(g_k^2P(a)\right) = \psi(P(a))$ since the sequences
$\{g_kP(a)\}$ and $\left\{g_k^2P(a)\right\}$ both converge increasingly to $P(a)$. It follows that 
$$
\lim_{k \to \infty}
\psi\left((g_k-1)P(a)(g_k-1)\right) = 0
$$ 
and the estimate above then
shows that $\lim_{k \to \infty} \psi(g_kag_k) = \psi(a)$. By using
this conclusion in combination with (\ref{e19}) it follows that
$\psi(a) = \psi\circ P(a)$ when $a \geq 0$ and $\psi \circ P(a) < \infty$. Since
$\psi \circ P$ is a proper weight we can
now use Corollary 1.15 in \cite{KV1} to conclude that $\psi = \psi
\circ P$. 
%(It should be noted that \cite{KV2} does not contain the
%proof of this proposition. The proof can be found in \cite{KV1}.) 
\end{proof}

By taking $c=c_0$ in the theorem we obtain the following corollary.

\begin{cor}\label{e22} Let $c : \mathcal G \to \mathbb R$ be a
  homomorphism such that $\ker c \cap \mathcal G_x^x = \{x\}$ for all
  $x\in \mathcal G^{(0)}$ and $\ker c$ is open in $\mathcal G$. 

The map $m \mapsto \varphi_{m}$ is a bijection from the $(\mathcal G,c)$-conformal measures $m$ on $\mathcal G^{(0)}$
  with exponent $\beta$ onto the set of $\beta$-KMS weights for the action
  $\sigma^c$ on $C^*_r(\mathcal G)$.
\end{cor}

The bijection in Corollary \ref{e22} between measures and weights restricts of course
to a bijective correspondence between KMS states and
$(\mathcal G,c)$-conformal Borel probability measures. For those the
corollary is a consequence of Theorem 1.3 in \cite{N}, and the condition
that $\ker c$ is open is not necessary for states. But without the condition that $\ker c \cap \mathcal G_x^x$ is
trivial there can be KMS states that do not factor through the
conditional expectation $P$ and are not determined by their
restriction to $C_0(\mathcal G^{(0)})$. In \cite{N} these additional KMS states are
given a general description in terms of fields of
traces.

It is not inconceivable that Corollary \ref{e22}
remains true without the condition that $\ker c$ is open; this is left
as an unsolved problem. Theorem \ref{e1} and Corollary \ref{e22} suffice for the actions on graph
algebras we study in this paper.

\section{KMS weights on the $C^*$-algebra of a local homeomorphism}

\subsection{The RDA-groupoid of a local homeomorphism}

Now we turn to the case where $\mathcal G$ is the RDA-groupoid $\Gamma_{\phi}$ of a local
homeomorphism $\phi$ on a locally compact second countable Hausdorff space $X$. To
introduce this construction set
$$
\Gamma_{\phi} = \left\{ (x,k,y) \in X \times \mathbb Z  \times X :
  \ \exists n,m \in \mathbb N, \ k = n -m , \ \phi^n(x) =
  \phi^m(y)\right\} .
$$
This is a groupoid with the set of composable pairs being
$$
\Gamma_{\phi}^{(2)} \ =  \ \left\{\left((x,k,y), (x',k',y')\right) \in \Gamma_{\phi} \times
  \Gamma_{\phi} : \ y = x'\right\}.
$$
The multiplication and inversion are given by 
$$
(x,k,y)(y,k',y') = (x,k+k',y') \ \text{and}  \ (x,k,y)^{-1} = (y,-k,x)
.
$$
Note that the unit space of $\Gamma_{\phi}$ can be identified with
$X$ via the map $x \mapsto (x,0,x)$. Under this identification the
range map $r: \Gamma_{\phi} \to X$ is the projection $r(x,k,y) = x$
and the source map the projection $s(x,k,y) = y$.

To turn $\Gamma_{\phi}$ into a locally compact topological groupoid, fix $k \in \mathbb Z$. For each $n \in \mathbb N$ such that
$n+k \geq 0$, set
$$
{\Gamma_{\phi}}(k,n) = \left\{ \left(x,l, y\right) \in X \times \mathbb
  Z \times X: \ l =k, \ \phi^{k+n}(x) = \phi^n(y) \right\} .
$$
This is a closed subset of the topological product $X \times \mathbb Z
\times X$ and hence a locally compact Hausdorff space in the relative
topology.
Since $\phi$ is locally injective $\Gamma_{\phi}(k,n)$ is an open subset of
$\Gamma_{\phi}(k,n+1)$ and the union
$$
{\Gamma_{\phi}}(k) = \bigcup_{n \geq -k} {\Gamma_{\phi}}(k,n) 
$$
is a locally compact Hausdorff space in the inductive limit topology. The disjoint union
$$
\Gamma_{\phi} = \bigcup_{k \in \mathbb Z} {\Gamma_{\phi}}(k)
$$
is then a locally compact Hausdorff space in the topology where each
${\Gamma_{\phi}}(k)$ is an open and closed set. In fact, as is easily verified, $\Gamma_{\phi}$ is a locally
compact groupoid in the sense of \cite{Re1} and an \'etale groupoid,
i.e. the range and source maps are local homeomorphisms.

By construction there is a canonical periodic one-parameter group acting
on $C^*_r(\Gamma_{\phi})$, given by the homomorphism $c_0$ on
$\Gamma_{\phi}$ defined such that
$c_0 (x,k,y) = k$. The \emph{gauge action} $\gamma$ on
$C^*_r(\Gamma_{\phi})$ is the corresponding
one-parameter group $\gamma =
\sigma^{c_0}$. That is,
$$
\gamma_t(f)(x,k,y) = e^{ikt} f(x,k,y) 
$$
when $f \in C_c\left(\Gamma_{\phi}\right)$. Since $c_0$ clearly
satisfies the two conditions required to apply Theorem \ref{e1} we
obtain the following.

\begin{prop}\label{c35} Let $c : \Gamma_{\phi} \to \mathbb R$ be a
  continuous homomorphism. There is a bijective correspondence between the
  $(\Gamma_{\phi}, c)$-conformal measures $m$ on $X$ with exponent
  $\beta$ and the
  gauge-invariant $\beta$-KMS weights for the
  action $\sigma^c$ on
  $C^*_r\left(\Gamma_{\phi}\right)$. The gauge-invariant $\beta$-KMS
  weight $\varphi_m$ associated to $m$ is
$$
\varphi_m(\cdot) = \int_X P(\cdot) \ dm .
$$
\end{prop}

Let $F : X \to \mathbb R$ be a continuous function and define $c_F :
\Gamma_{\phi} \to \mathbb R$ such that
$$
c_F(x,k,y) = \lim_{n \to \infty} \left(\sum_{i=0}^{n+k}
F\left(\phi^i(x)\right) - \sum_{i=0}^n F\left(\phi^i(y)\right)\right) .
$$
Then $c_F$ is a continuous homomorphism. Following a terminology first
introduced in dynamical systems by Denker and Urbanski in \cite{DU} we
say that a Borel measure $m$ on $X$ is \emph{$e^{\beta F}$-conformal} when
\begin{equation}\label{d81}
m(\phi(A)) = \int_A e^{\beta F(x)} \ dm(x)
\end{equation}
for every Borel set $A \subseteq X$ such that $\phi : A \to X$ is injective.

\begin{lemma}\label{d80} Let $m$ be a regular Borel measure on $X$. Then $m$
  is $(\Gamma_{\phi},c_F)$-conformal with exponent $\beta$ if and only
  $m$ is $e^{\beta F}$-conformal.
\end{lemma}
\begin{proof} Assume first that $m$ is  $(\Gamma_{\phi},c_F)$-conformal
  with exponent $\beta$ and consider a Borel subset $A \subseteq X$
  such that $\phi$ is injective on $A$. Since $\phi$ is a local homeomorphism and $X$
  is second countable we can write $A$ as a countable disjoint union
$$
A = \sqcup_{i \in I} A_i
$$
of Borel sets with the property that for each $i$ there is an open set $U_i$
such that $A_i \subseteq U_i$ and $\phi : U_i \to
\phi(U_i)$ is a homeomorphism. For each $i$ and each open subset $V
\subseteq U_i$ the set $\left\{(x,1, \phi(x)) : \ x \in V \right\}$ is
an open bi-section $W$ in $\Gamma_{\phi}$ with $s(W) = \phi(V), \
r(W) = V$, and it follows therefore that
$$
m(\phi(V)) = \int_V e^{\beta c_F\left(x,1,\phi(x)\right)} \ dm(x) = \int_V e^{\beta F(x)} \ dm(x) .
$$
By additivity and regularity of $m$,
$$
m(\phi(A)) = \sum_{i \in I} m(\phi(A_i)) = \sum_i \int_{A_i} e^{\beta
  F(x)} \ dm(x) = \int_A e^{\beta F(x)} \ dm(x).
$$

Assume next that (\ref{d81}) holds whenever $A$ is a Borel subset on
which $\phi$ is injective. It follows by iteration that 
\begin{equation}\label{d82}
m(\phi^n(A)) = \int_A e^{\beta \sum_{i=0}^{n-1} F(\phi^i(x))} \ dm(x)
\end{equation}
for every Borel set $A \subseteq X$ such that $\phi^n : A \to X$ is
injective, $n = 1,2, 3, \dots$, and then by 'inversion' that
 \begin{equation}\label{d83}
m(B) = \int_A e^{\beta \left( \sum_{i=0}^{n-1} F(\phi^i(x)) - \sum_{j=1}^mF\left(\phi^{-j}\left(\phi^n(x)\right)\right)\right)  } \ dm(x)
\end{equation}
when $A$ and $B$ are Borel subsets of $X$ such that $\phi^n$ is
injective on $A$, $\phi^m$ is injective on $B$ and $\phi^n(A) =
\phi^m(B)$, where $\phi^{-j}$ denotes the inverse of $\phi^j :
\phi^{m-j}(B) \to \phi^m(B)$.

Consider then an open bi-section $W$ of
$\Gamma_{\phi}$. For $k \in \mathbb Z$ set $W_k = \Gamma_{\phi}(k)
\cap W$. By additivity of measures, to establish (\ref{c31}) it suffices to establish it for
$W_k$. As a further reduction note that by monotone continuity of measures and integrals it suffices to
establish (\ref{c31}) with $W_k$ replaced by the set
$$ 
W(k,n) = W \cap \Gamma_{\phi}(k,n).
$$
Since $W_{k,n}$ is a bi-section in
$\Gamma_{\phi}$ it follows that $\phi^{n+k}$ is injective on
$r\left(W(k,n)\right)$ and $\phi^n$ is injective on
$s\left(W(k,n)\right)$. In addition $\phi^{n+k}\left(r(W(k,n)\right) =
\phi^n\left(s(W(k,n)\right)$ and we can apply (\ref{d83}). This
establishes (\ref{c31}) because
$$
c_F\left(r_{W(k,n)}^{-1}(x)\right) = \sum_{i=0}^{n+k-1} F(\phi^i(x)) -
\sum_{j=1}^nF\left(\phi^{-j}\left(\phi^n(x)\right)\right)
$$
when $x \in r(W(k,n))$.
 
\end{proof}

For homomorphisms $\Gamma_{\phi} \to \mathbb R$ of the form ${c_F}$ we have now proved the following reformulation of Proposition \ref{c35}.

\begin{prop}\label{c84} There is a bijective correspondence between
  the regular $e^{\beta F}$-conformal measures $m$ on $X$ and the
  gauge-invariant $\beta$-KMS weights for the
  action $\sigma^{c_F}$ on
  $C^*_r\left(\Gamma_{\phi}\right)$. The gauge-invariant $\beta$-KMS
  weight $\varphi_m$ associated to $m$ is
$$
\varphi_m(\cdot) = \int_X P(\cdot) \ dm .
$$
\end{prop}

If $F$ is either strictly positive or strictly negative, the equality $\ker c_F \cap
\left(\Gamma_{\phi}\right)^x_x = \{x\}$ holds for all $x \in X$, and it
follows then from \cite{N} that the $\beta$-KMS states are all
gauge-invariant. When $F$ is also locally constant it follows from
Corollary \ref{e2}, applied with $c= c_F$, that the same is true for $\beta$-KMS weights, and
Proposition \ref{c84} is then true with the word 'gauge-invariant'
deleted.

\begin{remark} Since $\phi$ is a local homeomorphism the Ruelle
  operator $L_g : C_c(X) \to C_c(X)$ is defined for every continuous
  and bounded
  real-valued function $g$ on $X$; viz.
$$
L_g(f)(x) = \sum_{y\in \phi^{-1}(x)} e^{g(y)} f(y) .
$$
The dual operator $L_g^*$ acts on the regular measures in the natural
way:
$$
L_g^*(m)(f) = \int_X L_g(f) \ dm.
$$
The Ruelle operator makes it possible to present an integrated version
of the identity (\ref{d81}) defining conformality of a
measure. Specifically, it is not difficult to see that a regular Borel
measure $m$ on $X$ is $e^{\beta F}$-conformal if and only if
$L_{-\beta F}^*(m) = m$, cf. \cite{Re2}. Therefore Proposition \ref{c84} is a
generalization to weights of Theorem 6.2 in \cite{Th2}. 
    
\end{remark}

\section{KMS weights on graph $C^*$-algebras}

\subsection{KMS weights for the generalized gauge actions on graph algebras}

Graph algebras were introduced in \cite{KPRR} and we shall
adopt notation and terminology from \cite{KPRR}. In particular, we assume that the
directed graph $G$ under consideration has at most countably many
vertexes and edges, and that it is 'row-finite', meaning that the number
of edges emitted from any vertex is finite. Furthermore, we assume
that there are no sinks, i.e. every vertex emits an edge.

Let $V_G$ and $E_G$ denote the set of vertexes and edges in $G$,
respectively. An \emph{infinite path} in $G$ is an element $p  \in
\left(E_G\right)^{\mathbb N}$ such that $r(p_i) = s(p_{i+1})$ for all
$i$, where we let $r(e)$ and $s(e)$ denote the range and source of an
edge $e \in E_G$, respectively.  A finite path $\mu = e_1e_2 \dots e_n$ is
defined similarly, and we extend the range and source maps to finite
paths such that $s(\mu) = s(e_1)$ and
$r(\mu) = r(e_n)$. The number of edges in $\mu$ is its \emph{length}
and we denote it by $|\mu|$. We let $P(G)$ denote the set of
infinite paths in $G$ and extend the source map to $P(G)$ such that
$s(p) = s(p_1)$ when $p = (p_i)_{i=1}^{\infty}$. To describe the topology of $P(G)$, let $\mu =
e_1e_2\cdots e_n$ be a finite path in $G$. We can then consider the \emph{cylinder} 
$$
Z(\mu) = \left\{p \in P(G) : \ p_i = e_i, \ i =1,2, \dots,n
\right\}.
$$
$P(G)$ is a totally disconnected second countable locally compact
Hausdorff space in the topology for which the collection of cylinders is
a base, \cite{KPRR}. The graph $C^*$-algebra $C^*(G)$ is then the (reduced) groupoid $C^*$-algebra $C^*_r\left(\Gamma_{\sigma}\right)$ of the RDA-groupoid $\Gamma_{\sigma}$ corresponding to the local homeomorphism $\sigma$ on the path space $P(G)$, defined such that
$$
\sigma\left(e_0e_1e_2 e_3 \cdots\right) = e_1e_2e_3 \cdots .
$$

Let $F : P(G) \to \mathbb R$ be a continuous function. The
corresponding homomorphism $c_F : \Gamma_{\sigma} \to \mathbb R$ is
then given by
$$
c_F(p, k, q) =  \lim_{n \to \infty} \left( \sum_{i=0}^{n+k} F\left(\sigma^i(p)\right) - \sum_{i=0}^n F\left(\sigma^i(q)\right) \right), 
$$
and we can consider the continuous one-parameter action $\alpha^F = \sigma^{c_F}$ on
$C^*(G)$ defined such that
$$
\alpha^F_t(f)(z) = e^{itc_F(z)} f(z)
$$
when $f \in C_c\left(\Gamma_{\sigma}\right)$. It follows from
Proposition \ref{c84} that the gauge-invariant $\beta$-KMS weights for
$\alpha^F$ are in
one-to-one correspondence with the $e^{\beta F}$-conformal Borel
measures on $P(G)$. In order to obtain a complete description of these
measures we
assume now that $F$ is given by
a map $F_0 : V_G \to \mathbb R$ through the formula
$$
F(p) = F_0\left(s(p)\right),
$$
where $s(p) \in V_G$ is the source of the path $p \in
P(G)$. It follows in particular that $F$ and hence also $c_F$ is
locally constant.

\begin{remark}\label{Remark1}
The assumption that $F$ only depends on the initial vertex of a path
is not as restrictive as it may appear at first sight. If
instead $F$ depends on an initial part of the paths, say on the first
$k$ edges, it would still be possible to apply the results we obtain
here. To see how, define a new graph $H$ whose vertexes are the paths
of length $k$ in $G$ and with an edge from the path $\mu = e_1e_2\dots
e_k$ to $\mu' = e'_1e'_2\dots e'_k$ when $e_2e_3e_4 \cdots e_k =
e'_1e'_2e'_3 \cdots e'_{k-1}$. The resulting path-space $P(H)$ is then homeomorphic to $P(G)$ under a
shift-commuting homeomorphism and the function $F$ would on $P(H)$
only depend on the initial vertex of the paths.  
\end{remark}

We extend $F_0$ to a map on finite paths such that
$$
F_0(\mu) = \sum_{i=1}^n F_0(s(e_i)),
$$
when $\mu = e_1e_2 \cdots e_n$. For any vertex $v \in V_G$ we set
\begin{equation}\label{k11}
C_v = \left\{ p \in P(G) : \ s(p) = v \right\}.
\end{equation}

\begin{lemma}\label{f14} Let $m$ be a regular Borel measure on
  $P(G)$. Then $m$ is $(\Gamma_{\sigma}, c_F)$-conformal with exponent $\beta$ if
  and only if
\begin{equation}\label{c36}
m(Z(\mu)) = e^{-\beta F_0(\mu)} m\left(C_{r(\mu)}\right)
\end{equation}
for every cylinder $Z(\mu)$.
\end{lemma}
\begin{proof} To prove that condition (\ref{c31}) in the present setting implies (\ref{c36}), consider
  first an edge $e$ in $G$ and apply
  (\ref{c31}) to the bi-section
$$
W = \left\{(p, 1, \sigma(p)) : \ p\in Z(e) \right\}
$$
to conclude that (\ref{c36}) holds when $\mu = e$. We can then prove
(\ref{c36}) by induction in the length $|\mu|$ of $\mu$: Assume that
(\ref{c36}) holds when $|\mu| = n$. Let $\mu = e_1e_2\cdots e_{n+1}$ be
a path of length $n+1$. The set
$$
U = \left\{ (p,1,\sigma(p)) : \ p \in  Z(\mu)  \right\}
$$
is a bi-section in $\Gamma_{\sigma}$ such that $s\left(U\right) =
Z\left( e_2e_3 \cdots e_{n+1}\right)$ and $r( U) = Z(\mu)$. It follows
therefore from (\ref{c31}) that
$$
m(Z(\mu))e^{\beta F_0(s(e_1))}  =  m\left( Z\left( e_2e_3 \cdots e_{n+1}\right)\right) .
$$
By induction hypothesis $ m\left( Z\left( e_2e_3 \cdots
    e_{n+1}\right)\right) = \exp (-\beta \sum_{i=2}^{n+1} F_0(s(e_i)))
m\left(C_{r(e_{n+1})}\right)$, and it follows therefore that
$m\left(Z(\mu)\right) = e^{-\beta F_0(\mu)}m\left(C_{r(\mu)}\right)$,
as desired.

Conversely, assume that (\ref{c36}) holds. It follows
straightforwardly that 
$$
m\left(\sigma\left(Z(\mu)\right)\right) = e^{\beta F_0(s(\mu))}
m(Z(\mu))
$$
for every cylinder $Z(\mu)$, which means that (\ref{d81}) holds when
$A$ is a cylinder. Consider then an open subset $V$ of $P(G)$ such
that $\sigma: V \to P(G)$ is injective. Define Borel measures $\nu_1$
and $\nu_2$ on $V$ such that
$$
\nu_1(B) = m(\sigma(B)), \ \nu_2(B) = \int_B e^{\beta F(x)} \ dm(x) .
$$ 
Since $\nu_1$ and $\nu_2$ agree on cylinder sets it follows from
regularity that $\nu_1$ and $\nu_2$ agree
on all Borel subsets of $V$. In particular, (\ref{d81}) holds for $V$,
and it follows as in the first part of the proof of Lemma \ref{d80}
that $m$ is $e^{\beta F}$-conformal. The same lemma then says that
$m$ is $(\Gamma_{\sigma},c_F)$-conformal. 

\end{proof}

Let $A = \left(A_{vw}\right)$ be the adjacency matrix of $G$, i.e. for
$v,w \in V_G$ we set
$$
A_{vw} = \# \left\{ e \in E_G : \ s(e) = v,  \ r(e) = w \right\} .
$$
In the following we let $\mathbb R^{V_G}$ denote the vector space of
  all functions $\xi : V_G \to \mathbb R$. Since $G$ is row-finite the
  adjacency matrix defines a linear map $A : \mathbb R^{V_G} \to
  \mathbb R^{V_G}$ in the usual way.

\begin{prop}\label{b1} The association $m \mapsto m(C_v)$ is
  a bijective correspondence between $(\Gamma_{\sigma},c_F)$-conformal
  Borel measures with exponent $\beta$ and the non-zero
  elements $\xi \in \mathbb R^{V_G}$ which satisfy the
  conditions
\begin{enumerate}
\item[a)] $\xi_v \geq 0$ and
\item[b)] $\sum_{w \in V}  A_{vw} \xi_w = e^{\beta F_0(v)} \xi_v$ 
\end{enumerate} 
for all $v \in V_G$.
\end{prop}
\begin{proof} For any measure $m$ on $P(G)$ the vector $\xi_v
  = m\left(C_v\right)$ trivially satisfies a). Property b) follows
  from (\ref{c36}) since it implies that
$$
m(C_v) = \sum_{\{e\in E_G : \ s(e) = v\}} \ m(Z(e)) \ = \ \sum_{w \in V}
A_{vw} e^{-\beta F_0(v)}m(C_w) .
$$
If $m_1$ and $m_2$ are $(\Gamma_{\sigma},c_F)$-conformal
  measures such that $m_1(C_v) = m_2(C_v)$ for all $v \in V_G$ it
  follows from Lemma \ref{f14} that $m_1$ and $m_2$ agree on cylinder
  sets, and then by regularity on every Borel subset. To complete the proof we must show that
  any element $\xi \in \mathbb R^{V_G}$ which satisfies a) and b) comes from a
  $(\Gamma_{\sigma},c_F)$-conformal measure. To see this set
$$
m(Z(\mu)) = e^{-\beta F_0(\mu)} \xi_{r(\mu)}.
$$
It follows from condition b) that 
$$
\sum_{\left\{ e \in E_G : \ s(e) = r(\mu)\right\}} m(Z(\mu e)) = m(Z(\mu)).
$$
Thanks to this relation we can for each fixed vertex $v$ define a measure $m$ on
the algebra of sets generated by the cylinder sets in $C_v$ such that $m(C_v) = \xi_v$,
and then standard measure theory methods applies to show that $m$
extends to a regular Borel measure $m$ on $P(G)$, cf. e.g. \cite{Pa}. 
\end{proof}

\begin{thm}\label{THM} There is a bijective correspondence between the gauge-invariant
  $\beta$-KMS weights for $\alpha^F$ and non-zero
  elements $\xi \in \mathbb R^{V_G}$ satisfying conditions
  a) and b) of Proposition \ref{b1}.
\end{thm}
\begin{proof} Combine Proposition \ref{c35} with Proposition \ref{b1}.
\end{proof}

\begin{cor}\label{THMcor}  There is a bijective correspondence between
  the gauge-invariant $\beta$-KMS states for $\alpha^F$ and non-zero elements $\xi \in
 \mathbb R^{V_G}$ satisfying conditions
  a) and b) of Proposition \ref{b1} plus the additional condition that
\begin{enumerate}
%\item[a)] $\xi_v \geq 0$ for all $v\in V_G$, 
%\item[b)] $\sum_{w \in V}  A_{vw} \xi_w = e^{\beta F_0(v)} \xi_v$ for all
%  $v\in V_G$, and
\item[c)] $\sum_{v \in V_G} \xi_v = 1$. 
\end{enumerate} 
\end{cor}

When $F_0$ is either strictly positive everywhere or strictly negative
everywhere, both Theorem \ref{THM} and its corollary remain true
with the word 'gauge-invariant' deleted, simply because all
KMS weights are then automatically gauge-invariant. In that case
Corollary \ref{THMcor} has been obtained by Carlsen and Larsen by
different methods, \cite{Ca}.

\subsection{KMS weights for the gauge action on graph algebras}

In this section we specialize to the case where $F_0$ is constant $1$ which means that
the action $\alpha^F$ becomes the gauge action $\gamma$, defined such
that $\gamma_t(f)(p,k,q) = e^{ikt} f(p,k,q)$ when $f \in
C_c(\Gamma_{\sigma})$. In this case the previous results take the following form.

\begin{thm}\label{THM2} There is a bijective correspondence between
  the $\beta$-KMS weights for the gauge action on $C^*(G)$ and non-zero
  elements $\xi \in \mathbb R^{V_G}$ satisfying 
\begin{enumerate}
\item[a')] $\xi_v \geq 0$ and
\item[b')] $\sum_{w \in V_G}  A_{vw} \xi_w = e^{\beta} \xi_v$ 
\end{enumerate} 
for all $v \in V_G$.
\end{thm}

\begin{cor}\label{THMcor2}  There is a bijective correspondence between
  the $\beta$-KMS states for the gauge action on $C^*(G)$ and non-zero
  elements $\xi \in
 \mathbb R^{V_G}$ satisfying conditions a') and b') of Theorem
 \ref{THM2}
plus the additional condition c) of Corollary \ref{THMcor}.
%\begin{enumerate}
%\item[a)] $\xi_v \geq 0$ for all $v\in V_G$, 
%\item[b)] $\sum_{w \in V}  A_{vw} \xi_w = e^{\beta F_0(v)} \xi_v$ for all
%  $v\in V_G$, and
%\item[c)] $\sum_{v \in V_G} \xi_v = 1$. 
%\end{enumerate} 
\end{cor}

\subsection{KMS weights for the gauge action when the graph is
  cofinal}

Recall that $G$ is \emph{cofinal} when every vertex can reach every
infinite path, i.e. for all $v \in V_G$ and all $p \in P(G)$ there is
a finite path $\mu$ in $G$ such that $s(\mu) =v$ and $r(\mu) = s(p_k)$
for some $k \in \mathbb N$. Recall also that simplicity of $C^*(G)$
implies cofinality of $G$, and that the converse is almost also true,
cf. Proposition 5.1 in \cite{BPRS}. As shown in \cite{Th3} the theory of
positive eigenvalues and eigenvectors of non-negative matrices, which
is well known in the irreducible row-finite case, has a natural
generalization to a class of matrices which comprises the adjacency
matrix $A$ of $G$, provided $G$ is cofinal. This allows us here to
obtain a complete description of the possible $\beta$-values for
which there is a $\beta$-KMS weight for the gauge action on $C^*(G)$
when $G$ is cofinal, and also to obtain a description of the
corresponding eigenvectors of $A$ and in this way the regular measures on
$P(G)$ which correspond to such $\beta$-KMS weights, at least in
principle. 

To formulate the results, let $NW_G$ be the (possibly empty) set of
vertexes in $G$ that support a loop; i.e. $v \in NW_G$ if and only if
there is a finite path $\mu$ in $G$ such that $s(\mu) = r(\mu) =
v$. Assuming that $G$ is cofinal, the set $NW_G$ and the edges emitted
from any element of $NW_G$ constitute a (still possibly empty)
irreducible subgraph of $G$ which we call \emph{the non-wandering part
  of $G$}. It follows that if $v \in NW_G$, the
number
$$
\beta_0 = \log \left( \limsup_{n \to \infty}
  \left(A^n_{vv}\right)^{\frac{1}{n}}\right),
$$
defined with the convention that $\log \infty = \infty$, is independent of the choice of $v$.

\begin{thm}\label{betavalues}  Assume that
  $G$ is cofinal and let $\beta \in \mathbb R$.
\begin{enumerate}
\item[1)] Assume that the non-wandering part $NW_G$ is empty. There is
  a $\beta$-KMS weight for the gauge action on $C^*(G)$ for all $\beta \in \mathbb R$. 
\item[2)]  Assume that the non-wandering part $NW_G$ is non-empty and
  finite. There is a $\beta$-KMS weight for the gauge action on
  $C^*(G)$ if and only if $\beta = \beta_0$.
\item[3)]  Assume that the non-wandering part $NW_G$ is non-empty and
  infinite. There is a $\beta$-KMS weight for the gauge action on
  $C^*(G)$ if and only if $\beta \geq \beta_0$.
\end{enumerate}
\end{thm} 
\begin{proof} Combine Theorem \ref{THM2} with Theorem 2.7 in \cite{Th2}.
\end{proof}

When $\beta = \beta_0$ and $\sum_{n=0}^{\infty} A^n_{vv} e^{-n\beta_0}
= \infty$, which is automatic in case 2), the $\beta_0$-KMS weight is
unique up to scalar multiplication. In all other cases there can 
be more than one extremal $\beta$-KMS weight, reflecting that the positive
$e^{\beta}$-eigenvectors of $A$ are not generally a scalar multiply
of each other, cf. \cite{Th3}. In \cite{Th3} it is shown how to generalize the well-known
Poisson-Martin integral representation of the harmonic functions of a
countable state
Markov chain to give an integral representation of the
positive $e^{\beta}$-eigenvectors of $A$. This leads in particular
to a description of the extremal rays in the set of positive
$e^{\beta}$-eigenvectors for $A$, and hence also, at least in
principle, a description of the extremal $\beta$-KMS weights. We refer
to Theorem 3.9 and Corollary 3.10 in \cite{Th3}.

In many cases the summability condition c) of Corollary \ref{THMcor}
prevents the existence of $\beta$-KMS states, although there may be
many $\beta$-KMS weights. See the last section for three different
examples of this.  

\subsection{On the factor type of an extremal KMS weight for the gauge action}\label{k40}

Given a weight $\psi$ on a $C^*$-algebra $A$ there is a GNS-type
construction consisting of a Hilbert space $H_{\psi}$, a linear map
$\Lambda_{\psi} : \mathcal N_{\psi} \to H_{\psi}$ with dense range and
a non-degenerate representation $\pi_{\psi}$ of $A$ on $H_{\psi}$ such that
\begin{enumerate}
\item[$\bullet$] $\psi(b^*a) = \left<
    \Lambda_{\psi}(a),\Lambda_{\psi}(b)\right>, \ a,b \in \mathcal
  N_{\psi}$, and
\item[$\bullet$] $\pi_{\psi}(a)\Lambda_{\psi}(b) = \Lambda_{\psi}(ab), \
  a \in A, \ b \in \mathcal N_{\psi}$,
\end{enumerate} 
cf. \cite{Ku},\cite{KV1},\cite{KV2}.
A $\beta$-KMS weight $\psi$ on $A$ is \emph{extremal} when the only
$\beta$-KMS weights $\varphi$ on $A$ with the property that
$\varphi(a) \leq \psi(a)$ for all $a \in A_+$ are scalar
multiples of $\psi$, viz. $\varphi = s\psi$ for some $s > 0$.

\begin{lemma}\label{k20} Let $A$ be a separable $C^*$-algebra and
  $\alpha$ a continuous one-parameter group of automorphisms on
  $A$. Let $\psi$ be an extremal $\beta$-KMS weight for $\alpha$. Then
  $\pi_{\psi}(A)''$ is a factor.
\end{lemma}
\begin{proof} It follows from \cite{KV1} that $\psi$ extends to normal
  faithful semi-finite weight $\tilde{\psi}$ on $\pi_{\psi}(A)''$ such
  that $\tilde{\psi} \circ \pi_{\psi} = \psi$, and
  that $t \mapsto \alpha_{-{t\beta}}$ extends to a
  $\sigma$-weakly continuous action $\theta$ on $\pi_{\psi}(A)''$ which is the
  modular automorphism group on $\pi_{\psi}(A)''$ associated to
  $\tilde{\psi}$. Let $e$ be a non-zero central projection in
  $\pi_{\psi}(A)''$ . Since $\theta$ acts trivially on the center it follows that $\theta_t(e) = e$ and it is then
  straightforward to verify that 
$$
A_+ \ni a \mapsto \tilde{\psi}(e \pi_{\psi}(a))
$$
is $\beta$-KMS weight on $A$. Since $\psi$ is extremal there is an $s >
0$ such that $\tilde{\psi}(e \pi_{\psi}(a)) = s\psi(a) =
\tilde{\psi}(s\pi_{\psi}(a))$ for all $a \in A_+$. This implies that $(e
-s)\Lambda_{\psi}(a) = \Lambda_{\psi}(a)$ for all $a \in \mathcal
N_{\psi}$, and hence that $e = s = 1$.    
\end{proof}

The next aim will be to determine the Connes invariant
$\Gamma\left(\pi_{\psi}\left(C^*(G)\right)''\right)$, assuming
that $\psi$ is an extremal $\beta$-KMS weight for the gauge action on $C^*(G)$. For this note
that it follows from Section 2.2 in \cite{KV1} that $\psi$ extends to a normal semi-finite faithful weight
$\tilde{\psi}$ on $\pi_{\psi}\left(C^*(G)\right)''$ such that $\psi =
\tilde{\psi} \circ \pi_{\psi}$, and that the modular group on
$\pi_{\psi}(C^*(G))''$ corresponding to $\tilde{\psi}$ is the
one-parameter group $\theta$ on $\pi_{\psi}(C^*(G))''$ given by
$$
\theta_t = \tilde{\gamma}_{-\beta t},
$$
where $\tilde{\gamma}$ is the
$\sigma$-weakly continuous extension of $\gamma$ defined such that
$\tilde{\gamma}_t \circ \pi_{\psi} = \pi_{\psi} \circ \gamma_t$. To
simplify the notation in the following, we set $M=
\pi_{\psi}(C^*(G))''$ and let $N \subseteq M$ be the fixed point
algebra of $\theta$, viz. $N = M^{\theta}$. 
Since $\theta_t = \theta_{t + \frac{2\pi}{\beta}}$ we can
define an action $\tilde{\theta}$ by the circle such that
$$
\tilde{\theta}_{e^{i t}} = \theta_{\frac{t}{\beta}}
$$
for all $t \in \mathbb R$. Note that
$\tilde{\theta}\left(C^*(G)\right) = C^*(G)$ and that the Arveson
spectrum $\Sp(qMq)$ of the restriction of
$\tilde{\theta}$ to $qMq$ is a subset of $\mathbb Z$ for all
projections $q \in N$. Let $v$ be a vertex in $G$ and $1_v
\in C_c(P(G))$
the characteristic function of the set $C_v$ from (\ref{k11}). For
simplicity of notation we also write $1_v$ for the image
$\pi_{\psi}(1_v) \in N$. By combining Definition 2.2.1 in \cite{C1} with Lemma 3.4.3, Proposition 2.2.2 and Lemme
2.3.3 in \cite{C1} it follows that  
\begin{equation}\label{j12}
\Gamma(M) = \beta \bigcap_e \Sp(e1_vM1_ve)
\end{equation}
where $v\in V_G$ can be any vertex and where we take the intersection over all non-zero central projections
$e$ in $1_vN1_v$.

Let $\mathcal P(G)$ be the set of finite paths in $G$. For every $v \in V$ we set 
$$
\Delta_v = \left\{|\mu|-|\nu| : \ \mu, \nu \in \mathcal P(G), \
 s(\mu) = s(\nu)= v, \ r(\mu) = r(\nu) \right\}.
$$
Then $\bigcap_{v  \in V} \Delta_v$
is a subgroup of $\mathbb Z$ and there is a unique natural number
$d_G$ such that
$\bigcap_{v \in V} \Delta_p = d_G \mathbb Z$.

\begin{prop}\label{j50} Let $\psi$ be an extremal $\beta$-KMS weight
  for the gauge action on $C^*(G)$. Then $\pi_{\psi}(C^*(G))''$ is a
  hyper-finite factor such that
$$
\Gamma\left(\pi_{\psi}\left(C^*(G)\right)''\right) \subseteq \mathbb Zd_G
\beta .
$$
\end{prop}
\begin{proof} That $\pi_{\psi}(C^*(G))''$ is hyper-finite follows from the
  nuclearity of $C^*_r(G)$, cf. \cite{KPRR}. Consider a vertex $v
\in V_G$ and let $k
\in  \Sp(1_vM1_v)$. There is an element $a
\in 1_vM1_v$ such that 
$$
\int_{\mathbb T} \lambda^{-k} \tilde{\theta}_{\lambda}(a) \ d\lambda =
a \neq 0.
$$
The density of $1_vC_c(\Gamma_{\sigma})1_v$ in $1_vM1_v$ for the $\sigma$-weak
topology implies that there is an element $b \in 1_vC_c(\Gamma_{\sigma})1_v$
such that 
$$
\int_{\mathbb T} \lambda^{-k} \tilde{\theta}_{\lambda}(b) \ d\lambda
\neq 0.
$$
Then $f = \int_{\mathbb T} \lambda^{-k} \tilde{\theta}_{\lambda}(b) \
d\lambda \in 1_vC_c(\Gamma_{\sigma})1_v$ is an element such that
$\theta_{\lambda}(f) = \lambda^k f$ for all $\lambda \in \mathbb T$. It follows that $f$ is supported
in 
$$
\Gamma_{\sigma}(-k) \cap s^{-1}(C_v) \cap r^{-1}(C_v),
$$
which must therefore be non-empty. Let $(p,-k,p')$ be an element in
this set. Then $\sigma^{n-k}(p) = \sigma^n(p')$ for all large $n$. In
particular, for $n$ large enough $\mu = p'_1p'_2 \cdots p'_{n}$ and
$\nu = p_1p_2\cdots p_{n-k}$ are paths in $G$ such that $s(\mu) = s(\nu) =
v, \ r(\nu) = r(\mu)$, and $|\mu| - |\nu| = k$. Since $v$ was
arbitrary, $k \in d_G \mathbb Z$. This gives the
stated inclusion, thanks to (\ref{j12}).
\end{proof}

Let $\mathbb P$ be set of integers $d \in \mathbb Z$ with the
property that there is a non-empty hereditary set $H \subseteq V_G$ and natural numbers $M,L \in \mathbb N$ such that for
every path $\mu$ in $H$ of length $M$ there are paths $l_+$ and $l_-$
with lengths $|l_+| \leq L, \ |l_-| \leq L$, and $s(l_+) = s(l_-) =
s(\mu), \ r(l_+)=r(l_-) = r(\mu)$, such that
$$
d = |l_+| - |l_-|.
$$

If $G$ is cofinal the intersection of two non-empty
  hereditary subsets of $V_G$ is again non-empty and
  hereditary. Therefore the set $\mathbb P$ is a subgroup of $\mathbb Z$
  in this case, and we can define $d'_G \in \mathbb N$ as the unique
  natural number such that $\mathbb P = \mathbb Zd'_G$.

\begin{prop}\label{k23} Let $\psi$ be an extremal $\beta$-KMS weight
  for the gauge action on $C^*(G)$. Assume that $G$ is cofinal and
  that $G$ has uniformly bounded out-degree, i.e. $\sup_{v \in V_G} \#
  s^{-1}(v) < \infty$. Then
  $\pi_{\psi}(C^*(G))''$ is a hyper-finite factor such that
$$
\mathbb Z d'_G\beta \subseteq
\Gamma\left(\pi_{\psi}\left(C^*(G)\right)''\right) \subseteq \mathbb Zd_G\beta.
$$
\end{prop}
\begin{proof} Consider a vertex $v$ and a non-zero central projection
  $q$ in $1_vN1_v$. In view of Proposition \ref{j50} and (\ref{j12}) it
  suffices to show that $d'_G \in  \Sp(qMq)$. Since
  $d'_G \in \mathbb P$ there is a non-empty hereditary set $H \subseteq V_G$ and natural numbers $M,L \in \mathbb N$ such that for
every path $\mu$ in $H$ of length $M$ there are paths $l_{\pm}$ with lengths $|l_{\pm}| \leq L$, $s(l_{\pm}) =
s(\mu), \ r(l_{\pm})= r(\mu)$ such that
$d'_G = |l_+| - |l_-|$. It follows from Lemma 3.7 in \cite{Th3} that
because $G$ is cofinal and $H$ hereditary there is an $N \in \mathbb
N$ such that every path of length $\geq N$ emitted from $v$ terminates
in $H$.

Observe that the fixed point algebra $C^*(G)^{\gamma}$ of the gauge action is the
reduced groupoid $C^*$-algebra of the closed and open
sub-groupoid $\Gamma_{\sigma}(0)$ of $\Gamma_{\sigma}$. The corner $1_vC^*(G)^{\gamma}1_v$ is the reduced
groupoid $C^*$-algebra of
$$
R_v = \bigcup_{n \in \mathbb N} \left\{ (p,p') \in P(G) \times P(G) :
  \ s(p) = s(p') = v, \ p_i = p'_i, \ i \geq n \right\} .
$$
For $n \in \mathbb N$, let $P_v(n)$ be the set
of paths $\mu$ in $G$ of length $n$ such that $s(\mu) = v$. For every pair $(\mu,\mu') \in P_v(n) \times P_v(n)$, let $e^n_{\mu, \mu'} \in C_c(R_v)$ be the
characteristic function of the set
$$
\left\{ (p,p') \in Z(\mu) \times Z(\mu') : \ p_i = p'_i , \ i \geq n +1
\right\}.
$$ 
Then $\left\{ e^n_{\mu,\mu'}: \ \mu,\mu' \in P_v(n) \right\}$
generates a finite dimensional $*$-subalgebra $A_{n}$
of
$C_c(R_v)$ such that 
$$
1_v C^*(G)^{\gamma}1_v = \overline{\bigcup_{n \in \mathbb
  N} A_n}.
$$ 
It follows that $\bigcup_{n \in \mathbb
  N} A_n$ is dense in $1_vN1_v$
for the strong operator topology. (Here and in the following we
suppress the representation $\pi_{\psi}$, and regard $C^*(G)$ as
acting on $H_{\psi}$.) Since $0 < \psi(1_v) < \infty$ there is a normal state $\omega_v$ on $1_vM1_v$
such that 
$$
\omega_v(a) = \psi(1_v)^{-1}\tilde{\psi}(a).
$$
Note that $\omega_v$ is a trace on $1_vN1_v$. Let $\| \cdot \|_{v}$ be the corresponding
$2-$norm,
$$
\left\| a \right\|_v = \sqrt{\omega_v(a^*a)} ,
$$
$ a \in 1_vM1_v$. For any $\epsilon > 0$ there is an $n \in \mathbb N$
and an element $a\in A_n$ such that $\|a\| \leq 1$ and $\left\|q-a\right\|_v\leq
\epsilon$. We shall require, as we can, that $n \geq N$, and to specify
how small an $\epsilon$ we want, let $B$ be a uniform upper bound for
the out-degree in $G$, i.e. $\# s^{-1}(v) \leq B$ for all $v \in V_G$,
and set
$$
K_1 = \frac{\min\{e^{(M-L)\beta}, e^{M\beta} \}}{MB}, \  K_2 =
\min\{e^{-d'_G\beta}, 1 \} .
$$  
We choose $\epsilon > 0$ so small that $\epsilon < 10^{-2}$ and
\begin{equation*}
4\epsilon^{\frac{1}{4}}
+2\epsilon^{\frac{1}{4}}K_2 +2\epsilon^{\frac{1}{4}} K_1K_2 \\
 <  K_1K_2\omega_v(q).
\end{equation*}
 Let $U(A_{n})$ be the
  unitary group of $A_{n}$ and set
$$
b = \int_{U(A_{n})} uau^* \ du,
$$
where we integrate with respect to the Haar measure $du$ on $U(A_{n})$. Then
$b$ is in the center of $A_{n}$ and we have the estimate
$$
\left\|b - q\right\|_v \leq \int_{U(A_n)} \left\|uau^* - uqu^*\right\|_v \
du \leq \epsilon .
$$
Standard arguments, e.g. as in the proof of Lemma
 12.2.3 in \cite{KR}, shows that there is a central projection $p$ in
 $A_{n}$ such that 
\begin{equation}\label{j88}
\left\|p -q\right\|_v \leq
 2\epsilon^{\frac{1}{4}}.
\end{equation}
By definition of $A_{n}$ there is a finite set $F$ of
 vertexes in $G$ such that 
$$
p = \sum_{w \in F} p_w ,
 $$ 
where $p_w$ is the characteristic function of $
\left\{p \in P(G) : \
  s(p) = v, \ r(p_n) = w\right\}$. Note that $F \subseteq H$ since $n
\geq N$. For each $w \in F$ we can therefore choose a path $\mu(w)$ in $H$ such
that $s(\mu(w)) = w$, $\left|\mu(w)\right| = M$ and
\begin{equation}\label{k25}
\omega_v(p_{\mu(w)}) \geq \frac{\omega_v(p_w)}{MB},
\end{equation}
where $p_{\mu(w)}$ is the characteristic function of the set 
$$\left\{p \in P(G) : \
  s(p) = v, \ r(p_n) = w, \ p_{n+1}p_{n+2} \cdots p_{n+M} = \mu(w) \right\}.
$$
Note that $p_{\mu(w)} \leq
p_w$. For each $w$ there are paths $l^w_{\pm}$ such that $s(l^w_{\pm}) = w,
\ r(l^w_{\pm})  = r\left(\mu(w)\right)$, $\left|l^w_{\pm}\right| \leq L$ and $d'_G =
\left|l^w_+\right| - \left|l^w_-\right|$. Let $u_{\pm}(w) \in
C_c(\Gamma_{\sigma})$ be the characteristic function of the set
consisting of the elements $(p,M-\left|l^w_{\pm}\right|,p') \in
\Gamma_{\sigma}$ with the properties: 
\begin{enumerate}
\item[$\bullet$] $s(p) =
s(p')  =v$,
\item[$\bullet$] $r(p_n) = r(p'_n) = w$,
\item[$\bullet$] $p_i= p'_i, \ i = 1,2, \cdots, n$,
\item[$\bullet$] $p_{n+1}p_{n+2} \cdots p_{n+M} = \mu(w)$,
\item[$\bullet$] $ p'_{n+1}p'_{n+2} \cdots p'_{n+|l^w_{\pm}|}= l^w_{\pm}$,
\item[$\bullet$] $p_{n+M+i} = p'_{n+\left|l^w_{\pm}\right|+i}, \ i
  \geq 1$.
\end{enumerate}
Then $u_{\pm}(w)$ are partial isometries such that
\begin{enumerate}
\item[a)] $u_{\pm}(w)u_{\pm}(w)^* = p_{\mu(w)}$,
  $u_{\pm}(w)^*u_{\pm}(w) \leq p_w$,
\item[b)] $\gamma_{t}\left(u_{\pm}(w)\right) =
  e^{i(M -\left|l^w_{\pm}\right|)t} u_{\pm}(w)), \ t \in \mathbb R$, 
\end{enumerate}
By combining b) with the fact that
${\psi}$ is a $\beta$-KMS weight for $\gamma$, we find that
$$
\omega_v(u_+(w)^*u_-(w)qu_-(w)^*u_+(w)) = e^{-d'_G\beta}\omega_v(
qu_-(w)^*u_+(w)u_+(w)^*u_-(w)).
$$ 
It follows from a) that
$u_+(w)u_+(w)^*u_-(w) = u_-(w)$ and we get then the identity 
\begin{equation}\label{k29}
\omega_v(u_+(w)^*u_-(w)qu_-(w)^*u_+(w)) = e^{-d'_G\beta}\omega_v( q
u_-(w)^*u_-(w)).
\end{equation}
Set $u_{\pm} = \sum_{w\in F} u_{\pm}(w)$. It follows from a) that
$u_{\pm}$ are partial isometries such that
\begin{equation}\label{l4}
u_-^*u_- = \sum_{w \in F} u_-(w)^*u_-(w) \leq p.
\end{equation} 
By using that
${\psi}$ is a KMS weight it follows that
$\omega_v\left(u_+(w)^*u_-(w)qu_-(w')^*u_+(w')\right) = 0$ when $w\neq w'$, and by combining with (\ref{k29}) we find that
\begin{equation}\label{j89}
\omega_v(u_+^*u_-qu_-^*u_+) 
= e^{-d'_G\beta}\omega_v(qu_-^*u_-).
\end{equation}
By a similar reasoning we find by use of a) and (\ref{k25}) that
$$
\omega_v(u_-(w)^*u_-(w)) =
e^{(M-\left|l^w_-\right|)\beta}\omega_v(u_-(w)u_-(w)^*) \geq
K_1\omega_v(p_w).
$$ 
Summing over $w \in F$ gives that
\begin{equation}\label{k32}
\omega_v(u_-^*u_-) \geq K_1\omega_v(p). 
\end{equation}
We can now conclude that
\begin{equation*}\label{k28}
\begin{split}
&\omega_v(qu_+^*u_-qu_-^*u_+q) \geq
\omega_v\left(u_+^*u_-qu_-^*u_+\right) -
4\epsilon^{\frac{1}{4}} \ \ \ \ \ \text{(using (\ref{j88}))}\\
&\\
& = - 4\epsilon^{\frac{1}{4}} + 
e^{-d'_G \beta}\omega_v(qu_-^*u_-)
\ \ \ \ \ \text{(using (\ref{j89}))}
\\
%& =  - 4\epsilon^{\frac{1}{4}} + \sum_{w \in F}
%e^{-p_G\beta}\omega_v(qu_-(w)^*u_-(w))\\
%& \geq  - 4\epsilon^{\frac{1}{4}} + e^{-p_G\beta}\omega_v(qu_-^*u_-)  \\
& \\
& \geq  - 4\epsilon^{\frac{1}{4}}
-2\epsilon^{\frac{1}{4}}K_2 + K_2\omega_v(u_-^*u_-) \ \ \ \ \
\text{(using (\ref{j88}) and (\ref{l4}))} \\
& \\
& \geq  - 4\epsilon^{\frac{1}{4}}
-2\epsilon^{\frac{1}{4}}K_2  + K_1K_2 \omega_v(p) \ \ \ \ \ \text{(using (\ref{k32}))} \\
& \\
& \geq  - 4\epsilon^{\frac{1}{4}}
-2\epsilon^{\frac{1}{4}}K_2 - 2\epsilon^{\frac{1}{4}}K_1K_2  + K_1K_2\omega_v(q) \ \ \ \ \ \text{(using (\ref{j88})).}
\end{split}
\end{equation*}
Thanks to the choice of $\epsilon$ this implies that $qu_+^*u_-q \neq
0$. Since
$\tilde{\theta}_{\lambda}(qu_+^*u_-q) = \lambda^{-d'_G}qu_+^*u_-q$, we conclude
that $d'_G \in \Sp(qMq)$, as desired.
\end{proof}

When $G$ is a finite irreducible graph, and more generally when $G$ is
cofinal and $NW_G$
is finite, the two numbers $d'_G$ and $d_G$ are both equal to the
global period of $NW_G$. For a finite irreducible graph Proposition
\ref{k23} recovers that part of the results from \cite{O} which deals with the gauge
action. As will be shown in Example \ref{example2} and Example
\ref{example3} below, there are other cases where the two numbers
$d'_G$ and $d_G$ agree and where Proposition \ref{k23} therefore
determines the $\Gamma$-invariant of $\pi_{\psi}(C^*(G))''$. However,
it is also easy to construct examples of infinite irreducible graphs, as the one
presented in Example \ref{example1} below, where they differ. I have
no idea what the $\Gamma$-invariant is in such cases.

\section{Examples} 

The first two examples are intended to show that in some cases it is
quite easy to find all the solutions to the equations and hence
identify all KMS weights for the gauge action. The third example is
meant to show how methods and results from the theory of countable
state Markov chains in some cases can be used for the same purpose. A
common feature is the scarcity of the KMS states compared to KMS weights.

\begin{example}\label{example1}

Consider the following graph with labeled vertexes:

\bigskip

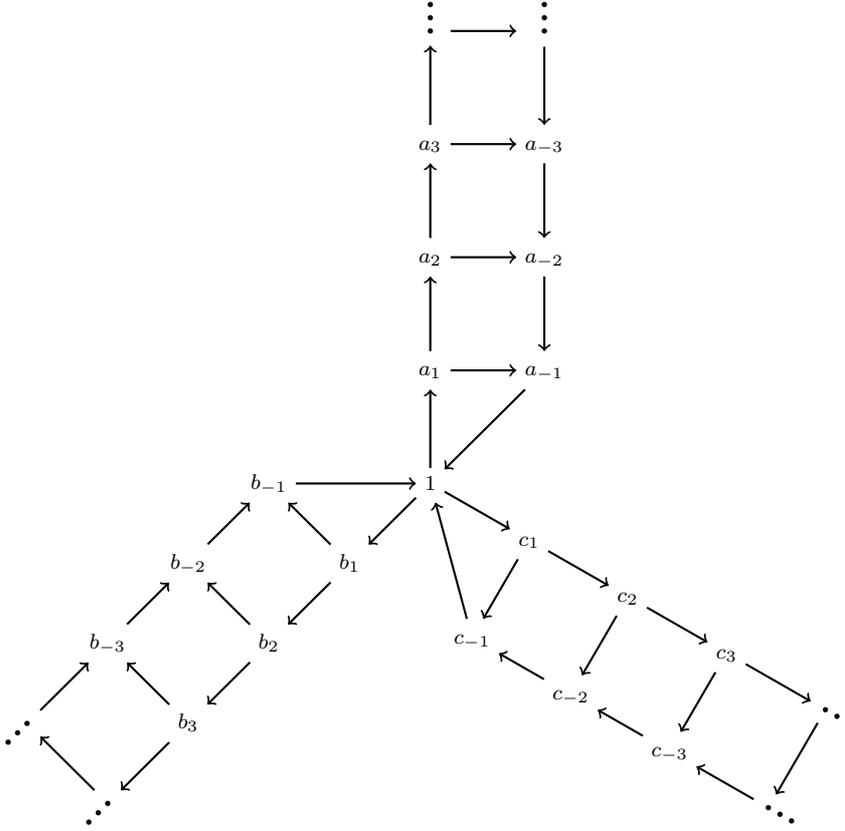
\begin{tiny}

\begin{tikzpicture}[->,thick,x=15mm,y=15mm]
  \begin{scope}
    \node (O) at (0,0) {$1$};
    \node (a1) at (0,1) {$a_1\mathstrut$};
    \node (a2) at (0,2) {$a_2\mathstrut$};
    \node (a3) at (0,3) {$a_3\mathstrut$};
    \node (a4) at (0,4) {$\phantom{a_4}$};
    \node (am1) at (1,1) {$a_{-1}\mathstrut$};
    \node (am2) at (1,2) {$a_{-2}\mathstrut$};
    \node (am3) at (1,3) {$a_{-3}\mathstrut$};
    \node (am4) at (1,4) {$\phantom{a_{-4}}$};

    \draw  (O)--(a1);
    \draw (a1)--(a2);
    \draw (a2)--(a3);
    \draw (a3)--(a4);
    \draw (a1)--(am1);
    \draw (a2)--(am2);
    \draw (a3)--(am3);
    \draw (a4)--(am4);
    \draw (am4)--(am3);
    \draw (am3)--(am2);
    \draw (am2)--(am1);
    \draw (am1)--(O);
    \fill 
    (a4) circle (1pt)
    (a4)++(0,5pt) circle (1pt)
    (a4)++(0,10pt) circle (1pt)
    (am4) circle (1pt)
    (am4)++(0,5pt) circle (1pt)
    (am4)++(0,10pt) circle (1pt)
    ;

%    \node at (O) [right=2mm] {$1$};
  \end{scope}
  \begin{scope}[rotate=-120]
    \node (O) at (0,0) {$\phantom{\mathstrut1}$};
    \node (c1) at (0,1) {$c_1\mathstrut$};
    \node (c2) at (0,2) {$c_2\mathstrut$};
    \node (c3) at (0,3) {$c_3\mathstrut$};
    \node (c4) at (0,4) {$\phantom{a}$};
    \node (cm1) at (1,1) {$c_{-1}\mathstrut$};
    \node (cm2) at (1,2) {$c_{-2}\mathstrut$};
    \node (cm3) at (1,3) {$c_{-3}\mathstrut$};
    \node (cm4) at (1,4) {$\phantom{a}$};
    
    \draw  (O)--(c1);
    \draw (c1)--(c2);
    \draw (c2)--(c3);
    \draw (c3)--(c4);
    \draw (c1)--(cm1);
    \draw (c2)--(cm2);
    \draw (c3)--(cm3);
    \draw (c4)--(cm4);
    \draw (cm4)--(cm3);
    \draw (cm3)--(cm2);
    \draw (cm2)--(cm1);
    \draw (cm1)--(O);
    \fill 
    (c4) circle (1pt)
    (c4)++(0,5pt) circle (1pt)
    (c4)++(0,10pt) circle (1pt)
    (cm4) circle (1pt)
    (cm4)++(0,5pt) circle (1pt)
    (cm4)++(0,10pt) circle (1pt)
    ;
  \end{scope}

  \begin{scope}[rotate=135]
    \node (O) at (0,0) {$\phantom{1\mathstrut}$};
    \node (b1) at (0,1) {$b_1\mathstrut$};
    \node (b2) at (0,2) {$b_2\mathstrut$};
    \node (b3) at (0,3) {$b_3\mathstrut$};
    \node (b4) at (0,4) {$\phantom{a}$};
    \node (bm1) at (1,1) {$b_{-1}\mathstrut$};
    \node (bm2) at (1,2) {$b_{-2}\mathstrut$};
    \node (bm3) at (1,3) {$b_{-3}\mathstrut$};
    \node (bm4) at (1,4) {$\phantom{a}$};
    
    \draw  (O)--(b1);
    \draw (b1)--(b2);
    \draw (b2)--(b3);
    \draw (b3)--(b4);
    \draw (b1)--(bm1);
    \draw (b2)--(bm2);
    \draw (b3)--(bm3);
    \draw (b4)--(bm4);
    \draw (bm4)--(bm3);
    \draw (bm3)--(bm2);
    \draw (bm2)--(bm1);
    \draw (bm1)--(O);
    \fill 
    (b4) circle (1pt)
    (b4)++(0,5pt) circle (1pt)
    (b4)++(0,10pt) circle (1pt)
    (bm4) circle (1pt)
    (bm4)++(0,5pt) circle (1pt)
    (bm4)++(0,10pt) circle (1pt)
    ;
  \end{scope}

\end{tikzpicture}

\end{tiny}

\bigskip

For this graph it is quite easy to see that a map $\xi : V_G \to
[0,\infty)$ which is normalized such
that $\xi_1 =1$, is a solution to
the equation b') of Theorem \ref{THM2} exactly when
\begin{enumerate}
\item[i)] $\xi_{a_1} + \xi_{b_1} + \xi_{c_1} = e^{\beta}$,
\item[ii)] $\xi_{a_{-n}} = \xi_{b_{-n}} = \xi_{c_{-n}} = e^{-\beta n},
 \  n = 1,2,3, \dots$, and
\item[iii)] $\xi_{a_{n+1}} + e^{-n\beta} = e^{\beta}\xi_{a_n}, \
  \xi_{b_{n+1}} + e^{-n\beta} = e^{\beta}\xi_{b_n}, \ \xi_{c_{n+1}} +
  e^{-n\beta} = e^{\beta}\xi_{c_n}, \ n \geq 1$ 
\end{enumerate}
It follows that
$$
\xi_{a_{n+1}} = e^{n\beta} \left(\xi_{a_1} - \sum_{j=1}^n
  \left(e^{-2\beta}\right)^j\right), \ n \geq 1 ,
$$
combined with similar formulas involving the $b_n$'s and $c_n$'s.
The positivity requirement on $\xi$ implies that $\beta > 0$ and that 
$$
\min \{\xi_{a_1}, \xi_{b_1},\xi_{c_1}\} \geq \sum_{j=1}^{\infty} (e^{-2\beta})^j =
\frac{e^{-2 \beta}}{1-e^{-2\beta}} .
$$   
Combined with condition i) it follows that 
$3 \frac{e^{-2\beta}}{1-e^{-2\beta}} \leq e^{\beta}$,
which means that $\beta \geq {\log \alpha} \sim 0,5138$, where $\alpha$ is the
real root of the polynomial $x^3 - x -3$. For $\beta = {\log
  \alpha}$ there is a unique solution, and hence there is a unique
$\beta$-KMS weight for the gauge action in this case, up to scalar multiplication. For all values of
$\beta > {\log \alpha}$ the set of $\beta$-KMS weights form a
cone with a
triangle as base. The extreme rays of the cone correspond to the three cases where
$$
\left\{\xi_{a_1},\xi_{b_1},\xi_{c_1} \right\} = \left\{\frac{e^{-2\beta}}{1-e^{-2\beta}}
  , \ e^{\beta}  - \frac{2e^{-2\beta}}{1-e^{-2\beta}}  \right\}.
$$   
Note that only for the unique solution with $\beta = {\log
  \alpha}$ is the sum $\sum_{v \in V} \xi_v$ finite. It follows
that only for this
value of $\beta$ is there a KMS state, and it is then unique.

Concerning the $\Gamma$-invariant of $\pi_{\psi}(C^*(G))''$ when
$\psi$ is an extremal $\beta$-KMS weight, we observe that in the
present example, $d_G = 1$ while $d'_G = 0$. Hence the results from
Section \ref{k40} tell us only that $\Gamma \left( \pi_{\psi}(C^*(G))''\right) \subseteq \mathbb Z \beta$.

\end{example}

\begin{example}\label{example2}

Consider the following graph, again with a convenient labeling of the vertexes:

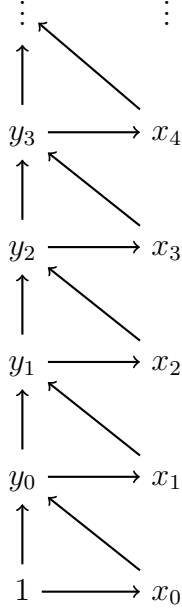
\begin{center}
\begin{tikzpicture}[->,thick]
% makes it easier for us, 'matrix of math nodes' already exists, but
% we'd like to add a twist
  \tikzstyle{matrix of math nodes}=[%
    matrix of nodes,
    nodes={%
     execute at begin node=$,%
     execute at end node=\mathstrut$% uniform natural height
    }%
  ]

  \matrix (M) [
  matrix of math nodes,
  column sep=3em,
  row sep=2em,
  ]{
    \vdots & \vdots \\
    y_3    & x_4    \\
    y_2    & x_3    \\
    y_1    & x_2    \\
    y_0    & x_1    \\
    1      & x_0    \\
  };

% horizontal
\draw (M-2-1) -- (M-2-2);
\draw (M-3-1) -- (M-3-2);
\draw (M-4-1) -- (M-4-2);
\draw (M-5-1) -- (M-5-2);
\draw (M-6-1) -- (M-6-2);

% vertical  
\draw (M-2-1) -- (M-1-1);
\draw (M-3-1) -- (M-2-1);
\draw (M-4-1) -- (M-3-1);
\draw (M-5-1) -- (M-4-1);
\draw (M-6-1) -- (M-5-1);

% angular
\draw (M-2-2) -- (M-1-1);
\draw (M-3-2) -- (M-2-1);
\draw (M-4-2) -- (M-3-1);
\draw (M-5-2) -- (M-4-1);
\draw (M-6-2) -- (M-5-1);

\end{tikzpicture}
\end{center}
%\begin{center}
%\begin{xymatrix}{
% \vdots &  \vdots\\
%y_3 \ar[r]  \ar[u] & x_4 \ar[ul]\\
%y_2 \ar[r]  \ar[u] &  x_3 \ar[ul] \\
% y_1 \ar[r] \ar[u]  &  x_2 \ar[ul] \\
%y_0\ar[r] \ar[u] & x_1 \ar[ul] \\
%1 \ar[r] \ar[u]&  x_0 \ar[ul] \\
%}
%\end{xymatrix}
%\end{center}

The solution $\xi$ to the eigenvalue equations a') and b') of Theorem
\ref{THM2} which is normalized to take the value 1 at the vertex in the
lower left hand corner,  is unique for all $\beta \in \mathbb R$ and is
given by
$$
\xi_{x_n} = \frac{e^{(2n+1)\beta}}{\left(1+e^{\beta}\right)^{n+1}}, \ \ \xi_{y_n}
= \left( \frac{e^{2\beta}}{1+e^{\beta}}\right)^{n+1}, \ n
=0,1,2,\cdots 
$$ 
Note that the corresponding $\beta$-KMS weight can be normalized to be
a state (i.e. equation c) of Theorem \ref{THMcor} can be made to hold) if and only
if $\beta < \log\left(\frac{1 + \sqrt{5}}{2}\right)$. Thus a
$\beta$-KMS state exists only for these values of $\beta$.

Concerning the $\Gamma$-invariant of $\pi_{\psi}(C^*(G))''$ when
$\psi$ is an extremal $\beta$-KMS weight, we observe that in the
present example, $d'_G = d_G = 1$. It follows therefore from
Proposition \ref{k23} that $\pi_{\psi}\left(C^*(G))\right)''$ is the
hyper-finite type $III_{\lambda}$-factor where $\lambda =
e^{-|\beta|}$, when $\beta \neq 0$. It is not difficult to see that
for $\beta = 0$ it is the
hyper-finite $II_{\infty}$ factor.

\end{example}

\begin{example}\label{example3} 
There are many results in the literature
which in specific cases can be used to find the solutions of the eigenvalue equations
a') and b') in Theorem \ref{THM2}. Most results about positive
eigenvalues and eigenvectors are motivated by applications to Markov
chains and therefore concerned with the case where the
matrix is stochastic in the sense that $\sum_{w \in V} A_{vw} = 1$ for all $v
\in V$, and with the case $\beta = 0$. The transition matrix of a
directed graph is rarely stochastic, but nonetheless these results are
highly relevant here since a given solution $\xi$ to a') and b') give
rise to the stochastic matrix
\begin{equation}\label{h1}
B_{vw} = e^{-\beta} \xi_v^{-1}\xi_w A_{vw},
\end{equation}
and the probabilistic methods and results for stochastic matrices can
then be used to find all other solutions. The following is an example
of this.

Let $\mu : \mathbb Z^d \to \mathbb N$
  be a finitely supported map. We define a matrix $A_{vw}, v,w \in
  \mathbb Z^d$, such that
$$
A_{vw} = \mu(w-v) .
$$ 
$A$ is the adjacency matrix of a graph $G$ with $\mathbb Z^d$ as vertex
set. Assume that the semi-group generated by the support of $\mu$ is all of
$\mathbb Z^d$. Then $A$ is irreducible (i.e. for all $v,w \in \mathbb
Z^d$ there is an $n \in \mathbb N$ such that $A^n_{vw} > 0$) and the
graph $G$ satisfies conditions in 3) of Theorem
\ref{betavalues}. We seek therefore first the value $\beta_0$; the smallest
$\beta$-value for which there is a positive $e^{\beta}$-eigenvector
for $A$, and for this the methods and results from Chapter 8
in \cite{Wo} can be adopted.

For every $c \in \mathbb R^d$, define $f_c: \mathbb Z^d \to
]0,\infty[$ such that
$$
f_c(v) = \exp(< c,v>), \ v \in \mathbb Z^d,
$$
where $< \cdot, \ \cdot>$ is the inner product in $\mathbb R^d$. It
is straightforward to check that $f_c$ is a positive
$e^{\beta}$-eigenvector for $A$, where
\begin{equation}\label{k41}
\beta = \log \left(\sum_{w \in \mathbb Z^d}
\mu(w)\exp(< c,w>)\right) .
\end{equation}
To see that the $f_c$'s generate all positive eigenvectors of $A$,
consider an arbitrary $\beta \in \mathbb R$ and let $\psi$ be a positive $e^{\beta}$-eigenvector for $A$, extremal
among those that take the value $1$ at $0$. For $u \in \mathbb Z^d$,
set $\psi_u(v) = \psi(u)^{-1}\psi(u+v)$. Then $\psi_u$ is also a
positive $e^{\beta}$-eigenvector. Note that 
$$
\psi(v) = \sum_{w \in \mathbb Z^d} e^{-\beta} \mu(w-v)\psi(w) =
\sum_{w \in \mathbb Z^d} e^{-\beta} \mu(w) \psi(w)\psi_{w}(v) .
$$
Since $\sum_{w \in \mathbb Z^d} e^{-\beta} \mu(w) \psi(w) = \psi(0) =
1$, the extremality of $\psi$ implies that $\psi_w = \psi$ for all $w$
in the support of $\mu$. That is, $\psi(w+v) = \psi(w)\psi(v)$ for all
$v \in \mathbb Z^d$ and all $w$ in the support of $\mu$. Since this
support generates $\mathbb Z^d$ by assumption, it follows that
$\psi(v+w) = \psi(v)\psi(w), v,w \in \mathbb Z^d$, which means that
there is a vector $c \in \mathbb R^d$ such that $\psi = f_c$. The
relation between $c$ and $\beta$ is given by (\ref{k41}). As shown in
\cite{Th3} the set of positive $e^{\beta}$-eigenvectors that are
normalized to take the value $1$ at $0$ constitute a compact Choquet
simplex, so it follows now that they are all contained in the closed convex
hull of a set of $f_c$'s. It follows
in particular from this that $e^{\beta_0}$ is the minimal value of
the strictly convex function 
$$
c \mapsto  \sum_{w \in \mathbb Z^d}
\mu(w)\exp(< c,w>)
$$  
on $\mathbb R^d$. Taking the gradient of this function we see that
$$
\beta_0 = \log \left( \sum_{w \in \mathbb Z^d}
\mu(w)\exp(< c_{min},w>)\right)
$$ 
where
$c_{min}$ is the unique solution to the equation
$$
\sum_{w \in \mathbb Z^d}
\mu(w)\exp(< c,w>)w = 0.
$$
By passing from $A$ to the matrix $B$ as in (\ref{h1}) we can apply a
result of Ney and Spitzer, stated as Theorem 8.15 in
\cite{Wo}, to conclude that the extremal rays in the cone of positive
$e^{\beta}$-eigenvectors are in one-to-one correspondence with the
points on the unit sphere $S^{d-1}$ in $\mathbb R^d$ for all $\beta
\geq \beta_0$, except when $\sum_{w \in \mathbb
  Z^d} \mu(w)w = 0$, in which case there is only one ray, i.e. the
positive $e^{\beta_0}$-eigenvector is unique up to scalar
multiplication.\footnote{ When $d \in \{1,2\}$ and $\sum_w \mu(w) w =
0$, Theorem 8.15 in \cite{Wo} does not apply because of the transience condition. In these cases the essential uniqueness of the positive
  $e^{\beta_0}$-eigenvector must be derived from the results of
  Vere-Jones, \cite{V}, dealing with the recurrent case.} In particular, the present example includes cases where the base
simplex for the cone of $\beta$-KMS weights of the gauge action is the
same for all $\beta \geq \beta_0$, namely the simplex of Borel
probability measures on $S^{d-1}$, as well as examples where this is
only the case for $\beta > \beta_0$ and the
simplex collapses to a point when $\beta$ hits its lowest possible value, as it
was the case in the first example.

Concerning KMS states we observe that in the present example there are
no $\beta$-KMS states at all. To see this consider a positive
$e^{\beta}$-eigenvector $\xi$ for some $\beta \geq \beta_0$. Then $e^{\beta}
= \sum_{w \in \mathbb Z^d} \mu(w)\exp(<c',w>)$ for some $c' \in
\mathbb R^d$ and it follows from Corollary 8.10 in \cite{Wo} that
$$
\xi_v = \int_{\mathbb R^d} \exp(<c+c',v>) \ d\mu(c)
$$
for some Borel measure $\mu$ on $\mathbb R^d$. Since $\sum_{v\in
  \mathbb Z^d} \exp(< c+c',v>) = \infty$ for all $c \in \mathbb R^d$, it follows that $\sum_{v\in \mathbb Z^d} \xi_v =  \infty$, i.e. there are no positive
$e^{\beta}$-eigenvectors for $A$ satisfying condition c) in Corollary \ref{THMcor2}.

Concerning the $\Gamma$-invariant of $\pi_{\psi}(C^*(G))''$ when
$\psi$ is an extremal $\beta$-KMS weight, we observe that because $A$
is translation invariant, in the sense that
$A_{vw} = A_{v+uw+u}$ for all $u,v,w \in \mathbb Z^d$, the two numbers $d'_G$ and $d_G$ from
Section \ref{k40} agree. It follows therefore from
Proposition \ref{k23} that $\pi_{\psi}\left(C^*(G))\right)''$ is the
hyper-finite type $III_{\lambda}$-factor where $\lambda =
e^{- d_G\beta}$ for all $\beta \geq \beta_0$, cf. \cite{C2}.

\end{example}

\end{document}